\documentclass[10pt]{article}

\RequirePackage{amsmath}
\RequirePackage{amsfonts}
\RequirePackage{amssymb}
\RequirePackage{hyperref}
\RequirePackage{multirow}
\RequirePackage{algorithm}
\RequirePackage{graphicx}
\RequirePackage{xcolor}
\usepackage{booktabs}
\allowdisplaybreaks

\usepackage{amsmath}
\usepackage{amssymb}
\usepackage{amsthm}
\usepackage{textcomp}
\usepackage{graphicx}
\usepackage{algorithmic}
\usepackage{dpr_fec}
\usepackage{capt-of}
\usepackage{dsfont}
\usepackage{threeparttable}

\usepackage{dpr_fec,ifthen,dsfont}
\hypersetup{colorlinks=true}

\newtheorem{theorem}{Theorem}[section]

\newtheorem{lemma}{Lemma}[section]

\usepackage{isetechreport}
\def\reportyear{26T}
\def\reportno{008}
\def\revisionno{0}
\def\originaldate{\today}
\def\revisiondate{\today}
\coraltrue
\cvcrfalse

\begin{document}

\title{Low-Order Explicit Hessian Imitation Method for Large-Scale Supervised Machine Learning}
\author{Yunlang Zhu\thanks{E-mail: yuza23@lehigh.edu}}
\author{Lingjun Guo\thanks{E-mail: lig423@lehigh.edu}}
\author{Zahra Khatti\thanks{E-mail: zak223@lehigh.edu}}
\author{Xiaoyi Qu\thanks{E-mail: xiq322@lehigh.edu}}
\author{Chia-Yuan Wu\thanks{E-mail: chw222@lehigh.edu}}
\author{Lara Zebiane\thanks{E-mail: lara.zebiane@lehigh.edu}}
\author{Frank E.~Curtis\thanks{E-mail: frank.e.curtis@lehigh.edu}}
\affil{Department of Industrial and Systems Engineering, Lehigh University}
\titlepage

\maketitle

\begin{abstract}
  An algorithm is proposed for solving optimization problems arising in neural network training for supervised learning. The unique feature of the algorithm is the use of an auxiliary loss, in addition to the original loss employed for model training. The purpose of the auxiliary loss is to provide a mechanism for creating a low-order Hessian-type approximation for the original loss. The proposed algorithm employs the resulting low-order second-derivative approximation terms in place of the second-order momentum terms (i.e., squared elements of the gradient of the loss function) in an overall scheme that has computational cost on par with an Adam-type approach. Whereas the squared elements of a gradient vector do not necessarily approximate second-order derivatives well, by careful construction of the auxiliary loss, second-order derivative-type approximations for the original loss can be computed and employed by the algorithm in an efficient manner. A convergence guarantee is provided for the proposed algorithm that is on par with guarantees available for similar stochastic diagonal-scaling methods. The results of numerical experiments show situations when the proposed algorithm outperforms Adam and other popular modern optimizers. 
\end{abstract}


\section{Introduction}\label{sec.introduction}

Stochastic diagonal-scaling methods \cite{BottCurtNoce2018}, such as Adam \cite{KingBa2014} and its variants, are leading methods for training neural networks for supervised learning tasks. The computations involved in these methods are motivated in various ways, from discrete-step to continuous-time gradient flow perspectives.  Common features of the most successful methods include averaging and/or momentum. Such strategies seem particularly effective in the context of stochastic-gradient-based methods where, for example, averaging has the effect of mitigating the potential poor effects of derivative estimates that have large errors. This is opposed to the literature on deterministic optimization, which is dominated by (quasi-)Newton strategies whose strengths are primarily derived by their use of (approximate) second-order derivatives. The benefits of second-order derivatives are clear once the optimization process is close to a local minimizer to which a fast rate of local convergence can be achieved. Second-order derivatives also often help to create better steps when far from a local minimizer.

Despite this difference in the literature on the most successful optimization methods for deterministic and stochastic optimization, it remains an important question of interest whether one can design a strategy for supervised learning that leverages both classes of approaches, namely, averaging and/or momentum strategies to mitigate the effects of derivative estimates with large errors and the benefits of second-order derivatives (or approximations of them).

\subsection{Contributions}

We propose a new stochastic diagonal-scaling method, nicknamed LEHI, that has per-iteration computational costs that are on par with the very popular Adam scheme. The unique feature of LEHI is the use of an auxiliary loss function, the purpose of which is to provide a mechanism for creating a low-order Hessian-type approximation of the original loss.  The resulting algorithm leverages averaging and momentum (namely, running averages of first- and second-order terms) and attempts to approximate second derivatives more explicitly than is done in an Adam-type scheme.  We motivate and state LEHI, provide a convergence rate guarantee for it, and provide the results of a significant set of numerical experiments to demonstrate situations in which the proposed method outperforms Adam and other schemes.  We highlight upfront that the convergence-rate guarantee for our proposed algorithm is not stronger than that of Adam; the point is that the algorithm has a comparable guarantee and can perform better in practice in various settings. We provide a review of the relevant literature in~\S\ref{sec.literature}.  Our proposed algorithm and a convergence-rate guarantee for it are presented in~\S\ref{sec.algorithm}.  The results of numerical experiments are provided in~\S\ref{sec.numerical} and concluding remarks are offered in~\S\ref{sec.conclusion}.

\section{Literature Review}\label{sec.literature}

Stochastic-gradient methods have their roots in the method of stochastic approximation proposed by \cite{RobbMonr51}; see also \cite{RobbSieg71} for further theoretical guarantees. Significant advances in the design of stochastic diagonal-scaling methods that build off of the stochastic gradient methodology came in the past two decades with methods such as Adagrad~\cite{DuchHazaSing2011} and RMSprop~\cite{TielHint12}. The primary motivation for these methods was to mitigate the effect of (stochastic) gradient components of large magnitude.  For training (deep) neural networks, these methods were relatively quickly superseded by Adam~\cite{KingBa2014}, which employs both running averages of gradient components and squared gradient components.  In recent years, numerous variants of Adam have been proposed, including NAdam \cite{Dozat2016}, AdamW \cite{loshchilov2017decoupled}, AMSGrad \cite{reddi2019convergence}, and RAdam \cite{liu2021varianceadaptivelearningrate}.  Our proposed algorithm, LEHI, can be viewed as another alternative of Adam designed with a unique motivation in mind, as described in the next section.

\section{Algorithm Development}\label{sec.algorithm}

Let us begin with a statement of a simplified version of the Adam algorithm, which we provide as Algorithm~\ref{alg.adam_paper}.  The version of the algorithm that we state involves a modified bias correction term. This version has been shown to offer comparable practical performance with the original Adam algorithm, and allows one to provide a simpler analysis of a convergence guarantee; see \cite{DefoBottBachUsun2022}.  We base our proposed algorithm on this simplified version, but it is straightforward to derive a variant of our algorithm that employs the same bias correction as the original Adam scheme.  Note that, for any positive integer $j$, we denote the set of integers up to and including $j$ as $\{1,\dots,j\} =: [j]$.

Given a smooth empirical loss function $f : \mathbb{R}^d \to \mathbb{R}$ defined with respect to trainable parameters stacked into a vector $w \in \mathbb{R}^d$, the algorithm is designed to solve
\begin{equation}\label{prob.f_simple}
  \min_{w \in \mathbb{R}^d}\ f(w).
\end{equation}
The sequence $\{m_k\}$ maintains a running weighted average of stochastic gradients, whereas the sequence $\{v_k\}$ maintains a running average of squared stochastic gradient elements, which in turn are used to scale the step.  The sequence $\{\alpha_k\}$ incorporates a learning rate $\alpha$ along with a bias correction in terms of the parameter pair $(\beta_1,\beta_2)$.  Notice that the formula that defines $\{\alpha_k\}$ ensures that this sequence is monotonically increasing, which turns out to be useful for the analysis of the method. The parameter $\epsilon$ is used to ensure that the step components are not scaled by excessively large values.

\begin{algorithm}[ht]
  \caption{Adam (simplified)}
  \label{alg.adam_paper}
  \begin{algorithmic}[1]
    \REQUIRE $w_1 \in \mathbb{R}^{d}$, $\beta_1 \in [0,1)$, $\beta_2 \in (\beta_1,1)$, $\alpha \in (0,\infty)$, and $\epsilon \in (0,\infty)$
    \STATE set $m_0 \gets 0 \in \mathbb{R}^{d}$
    \STATE set $v_0 \gets 0 \in \mathbb{R}^{d}$
    \FOR{all $k\in\mathbb{N}$}
      \STATE set stochastic gradient estimate $g_k \approx \nabla f(w_k)$
      \STATE set $m_{k,i} \gets \beta_1 m_{k-1,i} + g_{k,i}$ for all $i \in [d]$
      \STATE set $v_{k,i} \gets \beta_2 v_{k-1,i} + g_{k,i}^2$ for all $i \in [d]$
      \STATE set $\alpha_k \gets \alpha \frac{(1-\beta_1)\sqrt{1-\beta_2^k}}{\sqrt{1-\beta_2}}$
      \STATE set $w_{k+1,i} \gets w_{k,i} - \alpha_k \frac{m_{k,i}}{\sqrt{\epsilon + v_{k,i}}}$ for all $i \in [d]$
    \ENDFOR
  \end{algorithmic}
\end{algorithm}

A potential shortcoming of typical Adam-type schemes, such as that presented in Algorithm~\ref{alg.adam_paper}, is that it does not leverage knowledge of the composite structure of the overall loss that is typical of (deep) neural network training.  Let us now investigate such a structure in order to motivate LEHI.

\subsection{Composite Objective Function}\label{sec.composite}

Suppose that $f : \mathbb{R}^{d} \times \mathbb{R}^{d_x} \times \mathbb{R}^{d_y} \to \mathbb{R}^{}$ is, in fact, defined as a composite of a loss $\ell : \mathbb{R}^{q} \to \mathbb{R}^{}$ and a prediction function $p : \mathbb{R}^{d} \times \mathbb{R}^{d_x} \times \mathbb{R}^{d_y} \to \mathbb{R}^{q}$, i.e., for all $w \in \mathbb{R}^{d}$, $x \in \mathbb{R}^{d_x}$, and $y \in \mathbb{R}^{d_y}$ one has
\begin{equation*}
  f(w,x,y) = (\ell \circ p)(w,x,y) = \ell(p(w,x),y).
\end{equation*}
Here, one can think of $w$ as the trainable parameters of a (deep) neural network and $(x,y)$ as an input/output pair for training the network. For simplicity, let us expand upon this formulation with a single pair $(x,y)$.  In supervised learning problems more generally, $f$ is a large sum of such terms defined over a set of input/output pairs $\{(x_j,y_j)\}_{j=1}^N$.  All of our subsequent discussion is easily generalized to this setting, but for our purposes here we simplify it to a single pair $(x,y)$.

Assuming that both $\ell$ and~$p$ are sufficiently differentiable such that the composite function $f$ is twice continuously differentiable (meaning that its Hessian is symmetric for all $w \in \mathbb{R}^d$), the gradient function $\nabla_w f : \mathbb{R}^{d} \to \mathbb{R}^{d}$ and Hessian function $\nabla_{ww}^2 f : \mathbb{R}^{d} \to \mathbb{R}^{d \times d}$ are given by
\begin{align}
  \nabla_w f(w,x,y)
  =&\ \underbrace{\nabla_w p(w,x)}_{d \times q} \underbrace{\nabla_p \ell(p(w,x),y)}_{q \times 1} \label{eq.dw} \\
  \text{and}\quad \nabla_{ww}^2 f(w,x,y) 
  =&\ \underbrace{\nabla_w p(w,x)}_{d \times q} \underbrace{\nabla_{pp}^2 \ell(p(w,x),y)}_{q \times q}  \underbrace{\nabla_w p(w,x)^T}_{q \times d} \nonumber \\
  & \quad + \sum_{i \in [q]} \underbrace{\nabla_{ww}^2 p_i(w,x)}_{d \times d} \underbrace{\nabla_{p_i} \ell(p(w,x),y)}_{1 \times 1}. \label{eq.ddw}
\end{align}
Here, $\nabla_{p_i} \ell$ is the $i$th component of $\nabla_p \ell$.  Derivative-based algorithms for minimizing $f$ with respect to~$w$ use these quantities or approximations of them.  In particular, at a high level:
\begin{itemize}
  \item (stochastic) gradient methods use \eqref{eq.dw} only.
  \item (stochastic) Newton methods use \eqref{eq.dw} and \eqref{eq.ddw}.
  \item (stochastic) generalized Gauss-Newton uses \eqref{eq.dw} and approximates \eqref{eq.ddw} with only its first term.  This provides a reasonable approximation when, for each $i \in [q]$, $\nabla_{ww}^2 p_i(w,x) \approx 0$ and/or $\nabla_{p_i} \ell(p(w,x),y) \approx 0$, so the second term is small.  A typical justification for assuming that this is indeed the case in the context of, e.g., least-squares problems is that $\ell = \frac{1}{2}\|\cdot\|_2^2$ yields $\nabla_{p_i} \ell = p_i(w,x) \approx 0$ when one is near a point with $p = 0$.  Otherwise, it may still be a reasonable assumption in the context of neural network training when one assumes that, near a solution, the function $p$ is \emph{linear-like} so that $\nabla_{ww}^2 p_i(w,x) \approx 0$ for all $i \in [q]$.
  \item (stochastic) diagonal-scaling methods like Adam form scaling matrices by averaging the diagonal elements of outer-product terms of the form
  \begin{align*}
    \nabla_w f(w,x,y) \nabla_w f(w,x,y)^T
    = \nabla_w p(w,x) H_{p,\ell}(w,x,y) \nabla_w p(w,x)^T,
  \end{align*}
  where the inner term is defined as
  \begin{equation*}
    H_{p,\ell}(w,x,y) \equiv \nabla_p \ell(p(w,x),y) \nabla_p \ell(p(w,x),y)^T.
  \end{equation*}
  Notice that this outer-product term only approximates the first term in~\eqref{eq.ddw} when
  \begin{equation}\label{eq.good_approx}
    H_{p,\ell}(w,x,y) \approx \nabla_{pp}^2 \ell(p(w,x),y).
  \end{equation}
  An assumption that these matrices are close to each other would rely on behavior of the loss $\ell$, not the prediction function $p$.  Assuming that \eqref{eq.good_approx} holds is not reasonable in general.
\end{itemize}

\subsection{Canonical Problem Type}\label{sec.canonical}

Going forward, let us suppose that the Hessian~\eqref{eq.ddw} has
\begin{align}
  \nabla_{ww}^2 f(w,x,y)
  \approx&\ \nabla_w p(w,x) \nabla_{pp}^2 \ell(p(w,x),y) \nabla_w p(w,x)^T, \label{eq.ass} \\
  \text{i.e.},\ \ 0 \approx&\ \sum_{i \in [q]} \nabla_{ww}^2 p_i(w,x) \nabla_{p_i} \ell(p(w,x),y). \nonumber
\end{align}
As we have seen in the context of (stochastic) generalized Gauss-Newton, the approximation in~\eqref{eq.ass} could be justified by assuming that $\nabla_{p_i} \ell(p(w,x),y) \approx 0$ for all $i \in [q]$, at least once a solution is approached.  However, the idea that $\nabla_p \ell(p(w,x),y) \approx 0$ might inherently contradict \eqref{eq.good_approx}, since in this case $\nabla_p \ell(p(w,x),y) \approx 0$ would suggest $\nabla_{pp}^2 \ell(p(w,x),y) \approx 0$, which is not something that should be assumed.  (For example, for a least-squares loss, this Hessian is the identity.)  Let us proceed instead with the idea that \eqref{eq.ass} can be justified by assuming that $\nabla_{ww}^2 p_i(w,x) \approx 0$ for all $i \in [q]$.

Now, the main idea of LEHI is that the middle term in~\eqref{eq.ass} can be approximated in a different manner than~\eqref{eq.good_approx}, since in general there is no reason to believe that the Hessian is well approximated by an outer product of the gradient vector.  Let us consider instead a Hessian approximation of the form
\begin{equation*}
  \nabla_w p(w,x) v(p(w,x),y) v(p(w,x),y)^T \nabla_w p(w,x)^T,
\end{equation*}
where $v(p(w,x),y) \in \mathbb{R}^q$. This approximation, with a rank-one outer product inside, leads to a method with cost comparable to Adam, say, when the second-order momentum terms are replaced by
\begin{equation}\label{eq.back}
  [\nabla_w p(w,x) v(p(w,x),y)]_i^2\ \ \text{for all}\ \ i \in [d].
\end{equation}
The main question is how to choose $v(p(w,x),y)$. Rather than $v(p(w,x),y) = \nabla_p \ell(p(w,x),y)$, as in Adam, one can instead choose $v(p(w,x),y)$ so as to ensure a better approximation of the form
\begin{equation}\label{eq.v}
  v(p(w,x),y) v(p(w,x),y)^T \approx \nabla_{pp}^2 \ell(p(w,x),y),
\end{equation}
which requires only an understanding of the explicit form of second-order derivatives of the loss $\ell$ with respect to~$p$.  For our proposed LEHI scheme, we propose to choose $v(p(w,x),y)$ as the gradient of an auxiliary loss function~$z$, where $z$ is chosen specifically with the aim to have~\eqref{eq.v}.  In this manner, \eqref{eq.back} can be computed using standard backpropagation.  We provide a couple of illustrative examples of how $z$ may be chosen for certain $\ell$ in Appendix~\ref{app.auxiliary}.  The following theorem, proved in Appendix~\ref{app.auxiliary}, shows an appropriate choice of an auxiliary loss for binary cross-entropy.

\begin{theorem}\label{th.exam}
  Consider the binary cross-entropy loss $\ell(p,y) = -y \log(\sigma(p)) - (1 - y) \log(1 - \sigma(p))$, where $\sigma(p) = (1 + e^{-p})^{-1}$.  Then, with the auxiliary loss $z(p,y) = \arcsin(\tanh(p/2))$ one finds that with $v(p,y) := \nabla_p z(p,y)$ one obtains $(v(p,y))^2 = \sigma(p)(1 - \sigma(p)) = \nabla_{pp}^2 \ell(p,y)$.
\end{theorem}

We emphasize that an appropriate auxiliary loss can be derived in a similar manner for any differentiable loss function. Theorem~\ref{th.exam} and those in Appendix~\ref{app.auxiliary} only offer some illustrative examples.

\subsection{Algorithm and Theorem}

Our proposed algorithm is stated as Algorithm~\ref{alg.lehi} below.  Due to its design, we refer to the algorithm as the Low-order Explicit Hessian Imitation (LEHI) Method.  As motivated in the preceding sections, one presumes that $f \equiv \ell \circ p$ and $\tilde{f} \equiv z \circ p$, and that $f$ is an empirical loss function over a dataset $\{(x_j,y_j)\}_{j=1}^N$ such that the gradients have the form
\begin{equation*}
  \begin{aligned}
    \nabla f(w_k) &= \frac1N \sum_{j=1}^N \nabla_w p(w_k,x_j) \nabla_p \ell(p(w_k,x_j),y_j) \\ \text{and}\ \ 
    \nabla \tilde{f}(w_k) &= \frac1N \sum_{j=1}^N \nabla_w p(w_k,x_j) \nabla_p z(p(w_k,x_j),y_j)
  \end{aligned}
\end{equation*}
Notice that this pair can be computed efficiently together with backprop due to their shared $\nabla_w p$ terms. The loss $z$ is chosen in such a manner that, along the diagonal elements, one has (recall \eqref{eq.v})
\begin{equation*}
  \begin{aligned}
    \nabla_p z(p(w,x),y) \nabla_p z(p(w,x),y)^T
    =:&\ v(p(w,x),y) v(p(w,x),y)^T \\
    \approx&\ \nabla_{pp}^2 \ell(p(w,x),y)\ \text{for all}\ \ (w,x,y) \in \mathbb{R}^d \times \mathbb{R}^{d_x} \times \mathbb{R}^{d_y}.
  \end{aligned}
\end{equation*}

\begin{algorithm}[ht]
  \caption{LEHI}
  \label{alg.lehi}
  \begin{algorithmic}[1]
    \REQUIRE $w_1 \in \mathbb{R}^{d}$, $\beta_1 \in [0,1)$, $\beta_2 \in (\beta_1,1)$, $\alpha \in (0,\infty)$, and $\epsilon \in (0,\infty)$
    \STATE set $m_0 \gets 0 \in \mathbb{R}^{d}$
    \STATE set $v_0 \gets 0 \in \mathbb{R}^{d}$
    \FOR{all $k\in\mathbb{N}$}
      \STATE set stochastic gradient estimate $g_k \approx \nabla f(w_k)$
      \STATE set stochastic gradient estimate $\tilde{g}_k \approx \nabla \tilde{f}(w_k)$
      \STATE set $m_{k,i} \gets \beta_1 m_{k-1,i} + g_{k,i}$ for all $i \in [d]$
      \STATE set $\vtilde_{k,i} \gets \beta_2 \vtilde_{k-1,i} + \tilde{g}_{k,i}^2$ for all $i \in [d]$
      \STATE set $\alpha_k \gets \alpha \frac{(1 - \beta_1) \sqrt{1 - \beta_2^k}}{\sqrt{1 - \beta_2}}$
      \STATE set $w_{k+1,i} \gets w_{k,i} - \alpha_k \frac{m_{k,i}}{\sqrt{\epsilon + \vtilde_{k,i}}}$ for all $i \in [d]$
    \ENDFOR
  \end{algorithmic}
\end{algorithm}

A convergence-rate guarantee for LEHI can be proved that is comparable to one proved for Adam in \cite{DefoBottBachUsun2022}. We prove this as Theorem~\ref{th.lehi} below, the proof of which is provided in Appendix~\ref{app.th.lehi}. The stochastic nature of the algorithm means that our analysis of it considers a stochastic process defined by the algorithm.  Let $(\Omega,{\cal F},\mathbb{P})$ denote a probability space that captures the behavior of the algorithm, i.e., each outcome in $\Omega$ represents a realization of a run of the algorithm.  Let $\mathbb{E}$ denote the expected value operator defined by the probability measure $\mathbb{P}$.  The only source of randomness in each iteration is the computation of the stochastic gradient estimates.  Let ${\cal F}_1$ be the $\sigma$-algebra defined by the initial conditions of an algorithm, and, more generally, for all $k \in \mathbb{N}$ let ${\cal F}_k$ be the $\sigma$-algebra generated by the initial conditions and the stochastic gradient estimators up through the end of iteration $k-1$.  In this manner, one has that ${\cal F}_1 \subseteq {\cal F}_2 \subseteq \cdots \subseteq {\cal F}$ and the sequence $\{{\cal F}_k\}$ is a filtration.

\begin{theorem}\label{th.lehi}
  Suppose that there exists an open convex set ${\cal W} \subseteq \mathbb{R}^d$ containing the iterates of any run of Algorithm~\ref{alg.lehi} over which $f$ is continuously differentiable and bounded below by $f_{\inf} \in \mathbb{R}$, and over which the gradient function $\nabla f : \mathbb{R}^d \to \mathbb{R}^d$ is Lipschitz continuous with constant $L \in (0,\infty)$.  In addition, suppose that for all $k \in \mathbb{N}$ one has
  \begin{equation}
    \mathbb{E}[g_k | {\cal F}_k] = \nabla f(w_k)
  \end{equation}
  and suppose that there exists a constant $M \in (0,\infty)$ such that for all $k \in \mathbb{N}$ one has
  \begin{equation}
    \max\{\|g_k\|_\infty,\|\tilde{g}_k\|_\infty\} \leq \sqrt{M^2 - \epsilon}.
  \end{equation}
  Finally, suppose that Algorithm~\ref{alg.lehi} is run for $K \in \mathbb{N}$ iterations, where $K > \frac{1}{1 - \beta_1}$, and let $R_K$ be a random variable taking each value in $\{1, \dots, K\}$ with probability equal to
  \begin{equation*}
    \mathbb{P}[R_K = k] = \frac{1-\beta_1^{K-k+1}}{\sum_{j=1}^K (1 - \beta_1^{K - j + 1})}.
  \end{equation*}
  Then, one has with $\tilde{K} = K - \frac{\beta_1}{1 - \beta_1}$ that
  \begin{equation*}
    \begin{aligned}
      \mathbb{E} [ \|\nabla f(w_{R_K}) \|_2^2 ]
      \leq&\ \frac{2M}{\alpha \tilde{K}} (f(w_1) - f_{\inf}) + E \left( \frac{1}{\tilde{K}} \log \left(1 + \frac{M^2}{(1-\beta_2)\epsilon} \right) - \frac{K}{\tilde{K}} \log(\beta_2) \right),
    \end{aligned}
  \end{equation*}
  where
  \begin{equation*}
    \begin{aligned}
      E :=&\ \frac{\alpha d M L (1 - \beta_1)}{(1 - \frac{\beta_1}{\beta_2})(1 - \beta_2)} + \frac{12 d M^2 \sqrt{1 - \beta_1}}{(1 - \frac{\beta_1}{\beta_2})^{3/2} \sqrt{1-\beta_2}} + \frac{2 \alpha^2 d L^2 \beta_1}{(1 - \frac{\beta_1}{\beta_2}) (1 - \beta_2)^{3/2}}.
    \end{aligned}
  \end{equation*}
\end{theorem}

The upper bound on the expected squared-norm of the gradient in Theorem~\ref{th.lehi} can be made as small as desired.  For a fixed $\beta_1$, one can choose $\beta_2$ sufficiently close to 1, then $\alpha$ sufficiently close to 0, and finally the iteration limit $K$ sufficiently large such that each term in the bound is as small as desired.

\section{Numerical Results}\label{sec.numerical}

We tested the performance of LEHI on a diverse set of tasks including regression, image classification, and language modeling using benchmark datasets. These tasks utilize a range of models including multilayer perceptrons (MLP), convolutional neural networks (CNNs), and transformers. We examined the training and testing performance of LEHI against Adam and AdamW. We also tested a hybrid algorithm, which we call LEHIBRID, that in every other iteration of LEHI replaces $\tilde{g}_k$ with $g_k$. (This essentially makes LEHIBRID a hybrid between LEHI and Adam.) For large language model training, we additionally included Sophia-G, a modern second-order optimizer \cite{ICLR2024_06960915}. To select the optimal learning rate for each algorithm for each problem, we adopted a confidence bound selection strategy. This ensures that the chosen hyperparameter yielded a solution that is both high-performing and stable. Mathematically, for a given metric $M$, we defined the selection score $S = \mu_K \pm c \cdot \sigma_K$, where $\mu_K$ and $\sigma_K$ are the mean and the standard deviation of $M$ over the last $K$ epochs in a run of the algorithm for the given problem, respectively. For the regression problem that we tested, the aim was to minimize the upper confidence bound of the testing loss ($S=\mu_K+  c \cdot \sigma_K$), for the  language modeling problem that we tested, the aim was to minimize the upper confidence bound of the validation loss ($S=\mu_K+  c \cdot \sigma_K$), and for the classification problems that we tested, the aim was to maximize the lower confidence bound of testing accuracy ($S=\mu_K- c \cdot \sigma_K$). In all cases, we used $c=2$ and averaged selection scores across three independent runs with random seeds $\{0,1,2\}$ to select the learning rate. The window $K$ is chosen at around the last $10\%$ of the number of epochs. Detailed reports for our learning rate search and final results can be found in Appendix \ref{app.sub.tuning} and \ref{app.sub.add}.



\subsection{UCI Protein}

UCI Protein is a low-dimensional regression problem with a moderately sized dataset \cite{physicochemical_properties_of_protein_tertiary_structure_265}. We trained an MLP with $9$ input neurons, $1$ hidden layer with $100$ neurons, and $1$ output neuron. ReLU activation was used at the hidden layer. We trained the model with $200$ epochs. For each optimizer, the best learning rate was selected based on the testing loss score across the final $10$ epochs. All other hyperparameters were held fixed: $\beta_1=0.9, \beta_2=0.999, \epsilon=10^{-7}$ for all optimizers, and \texttt{weight decay}$=10^{-2}$ for AdamW.

For this experiment, both LEHI (green) and LEHIBRID (purple) demonstrated consistent gains in convergence speed and final error reduction. In the training phase, our new methods utilized a higher learning rate of $0.1$ and reached the $ 2.4\times 10^{-1}$ loss threshold well before Adam (red) and AdamW (blue) (Figure \ref{fig:protein_loss}). The testing loss curves further validate these findings. On the testing set, LEHI and LEHIBRID consistently maintained a lower loss than the Adam baselines.

\begin{figure}[t]
  \centering
  \begin{tabular}{cc}
    \includegraphics[width=0.33\textwidth]{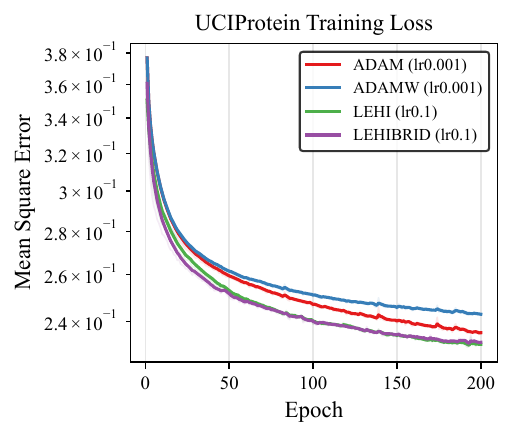} &
    \includegraphics[width=0.33\textwidth]{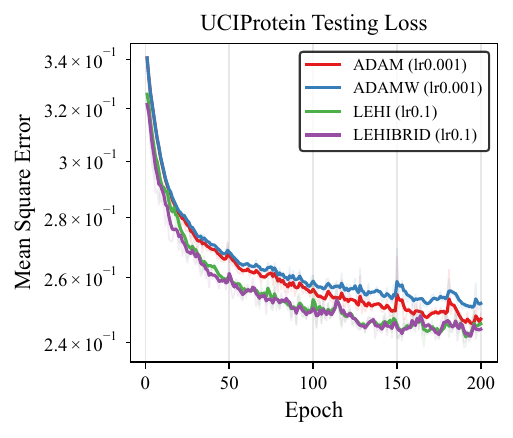} \\
    (a) Training loss & (b) Testing loss
  \end{tabular}
  \caption{\textbf{UCI Protein} loss curves. Learning rate was selected from $\{10^{-1}, 3\times10^{-2},10^{-2}, 3\times10^{-3}, 10^{-3}, 3 \times 10^{-4}, 10^{-4}\}$ based on the best testing performance. EMA smoothing ($\alpha=0.3$) is applied. Faded lines show raw curves.}
  \label{fig:protein_loss}
\end{figure}

\subsection{MNIST}
MNIST is a handwritten digit classification benchmark consisting of grayscale images of size $28 \times 28$ across $10$ classes \cite{1571417126193283840}. We trained an MLP with $28 \times 28$ input neurons, $1$ hidden layer with $50$ neurons, and $1$ output neuron. ReLU activation was used at the hidden layer. We trained the model with $25$ epochs. For each optimizer, the best learning rate was selected based on the testing accuracy score across the final $3$ epochs. All other hyperparameters were held fixed: $\beta_1=0.9, \beta_2=0.999, \epsilon=10^{-7}$ for Adam and AdamW, $\epsilon=10^{-2}$ for LEHI and LEHIBRID, and \texttt{weight decay}$=10^{-2}$ for AdamW. 

\begin{figure}[t]
  \centering
  \begin{tabular}{cc}
    \includegraphics[width=0.33\textwidth]{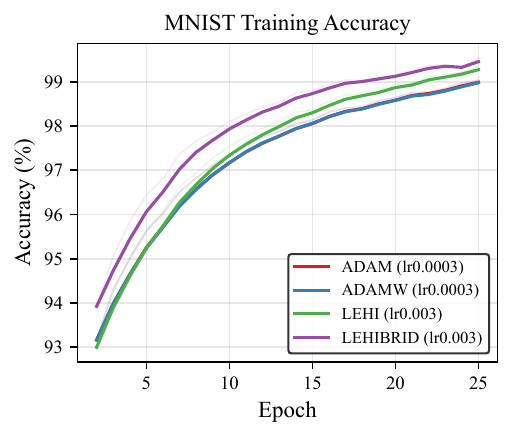} &
    \includegraphics[width=0.33\textwidth]{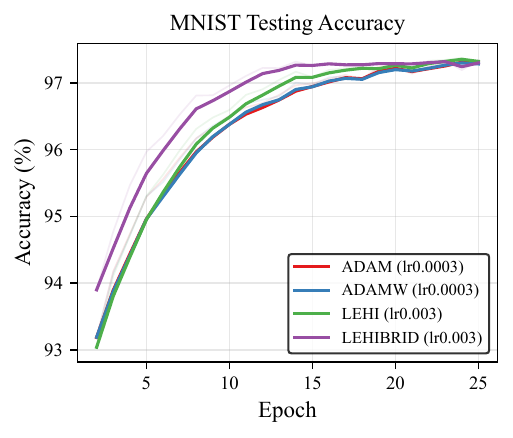} \\
    (a) Training accuracy & (b) Testing accuracy
  \end{tabular}
  \caption{\textbf{MNIST} accuracy curves. Learning rate was selected from $\{10^{-1}, 3\times10^{-2},10^{-2}, 3\times10^{-3}, 10^{-3}, 3 \times 10^{-4}, 10^{-4}\}$ based on the best testing performance. EMA smoothing ($\alpha=0.6$) is applied. First $2$ epochs are omitted to highlight differences. Faded lines show raw curves.}
  \label{fig:mnist_acc}
\end{figure}





In these experiments with the MNIST dataset, the training curves and the testing curves (Figure \ref{fig:mnist_acc}) exhibit similar trajectories. LEHIBRID (purple) demonstrated a clear advantage in both convergence speed and final accuracy compared to Adam (red) and AdamW (blue), maintaining a lead from the first epoch. Meanwhile, LEHI (green) overlapped with Adam and AdamW during the early epochs, but finally accelerated and surpassed the baseline methods.

\subsection{CIFAR-100}
CIFAR-100 extends CIFAR-10 (see CIFAR-10 experiments in Appendix \ref{app.sub.cifar10}) to a more challenging classification task with $100$ classes, while maintaining the same $32 \times 32$ image resolution \cite{krizhevsky2009learning}. We trained a ResNet-18 network using a minibatch size of $256$ with $150$ epochs. For each optimizer, the best learning rate was selected based on the testing accuracy score across the final $10$ epochs. All other hyperparameters were held fixed: $\beta_1=0.9, \beta_2=0.999$ for all optimizers, $\epsilon=10^{-7}$ for Adam and AdamW, $\epsilon=10^{-2}$ for LEHI and LEHIBRID, and \texttt{weight decay}$=10^{-2}$ for AdamW.

For these experiments with CIFAR-100, the training and testing accuracy trajectories (Figure \ref{fig:cifar100_acc}) exhibit trends comparable to those observed in the CIFAR-10 experiments. 
During the training process, although LEHIBRID (purple) and LEHI (green) initially showed a lower accuracy compared to Adam (red) and AdamW (blue), they demonstrated sustained learning throughout the 150-epoch duration. By approximately epoch 110, LEHI and LEHIBRID surpassed Adam and AdamW, ultimately achieving higher training accuracy than the baseline methods. On the testing set, the zoomed-in subplot further highlights that LEHI and LEHIBRID maintained a consistent lead from epoch $135$ to epoch $150$. In contrast, the baseline Adam and AdamW methods exhibited a relative plateauing in performance.

\begin{figure}[t]
  \centering
  \begin{tabular}{cc}
    \includegraphics[width=0.33\textwidth]{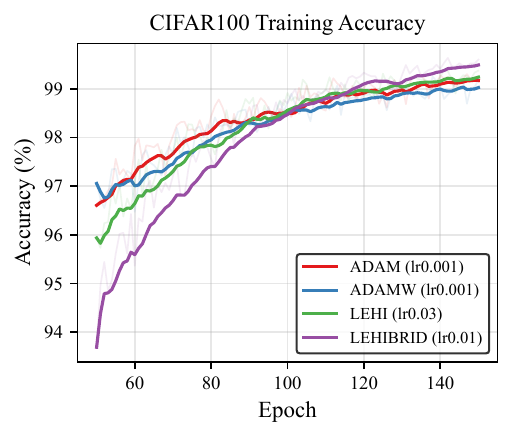} &
    \includegraphics[width=0.33\textwidth]{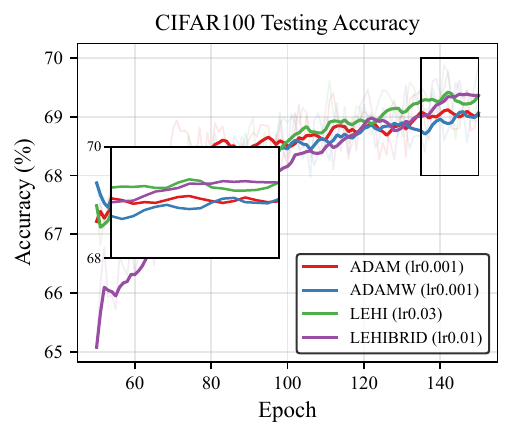} \\
    (a) Training accuracy & (b) Testing accuracy
  \end{tabular}
  \caption{\textbf{CIFAR100} accuracy curves. Learning rate was selected from $\{10^{-1}, 3\times10^{-2},10^{-2}, 3\times10^{-3}, 10^{-3}, 3 \times 10^{-4}, 10^{-4}\}$ based on the best testing performance. EMA smoothing ($\alpha=0.15$) is applied. First $50$ epochs are omitted to highlight differences. Faded lines show raw curves.}
  \label{fig:cifar100_acc}
\end{figure}

\subsection{Penn Treebank}
Penn Treebank (PTB) is standard for language modeling \cite{marcus-etal-1993-building}. Following common practice, we use the PTB corpus for next-token prediction, evaluating performance in an autoregressive language modeling setting. We trained a GPT-2 Small Transformer model from scratch using a fixed context length of $128$ tokens, and a minibatch size of $64$ with $25$ epochs. For each optimizer, the best learning rate was selected based on the validation loss (or log perplexity) score across the final $3$ epochs. All other hyperparameters were held fixed: $\beta_1=0.9, \beta_2=0.999$ for Adam and LEHI series, $\epsilon=10^{-7}$ for Adam and AdamW, $\epsilon=10^{-2}$ for LEHI and LEHIBRID, \texttt{weight decay}$=10^{-1}$ for AdamW, and $\beta_1=0.965,\beta_2=0.99,\epsilon=10^{-12},k=10,\rho=0.05$, \texttt{weight decay}$=0.2$ for Sophia-G.

\begin{figure}[t]
  \centering
  \begin{tabular}{cc}
    \includegraphics[width=0.33\textwidth]{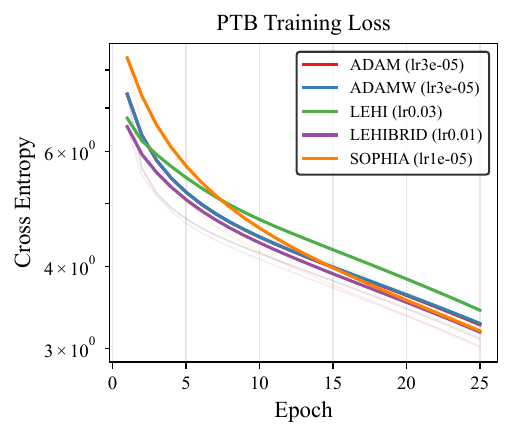} &
    \includegraphics[width=0.33\textwidth]{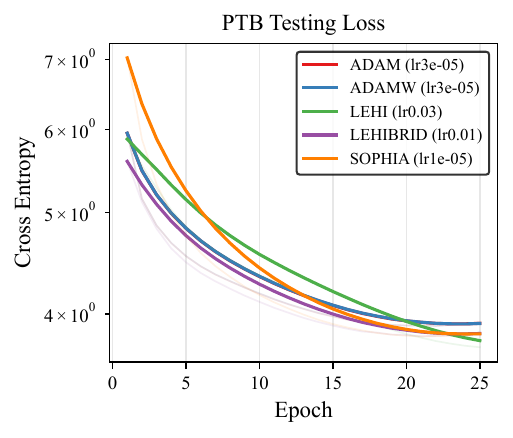} \\
    (a) Training loss & (b) Testing loss
  \end{tabular}
  \caption{\textbf{Penn Treebank} loss curves. Learning rate was selected from $\{10^{-1}, 3\times10^{-2}, 10^{-2}, 3\times10^{-3}, 10^{-3}, 3 \times 10^{-4}, 10^{-4}, 3 \times 10^{-5}, 10^{-5}, 3\times 10^{-6}\}$ based on the best validation performance. EMA smoothing ($\alpha=0.3$) is applied. First epoch is omitted. Faded lines show raw curves.}
  \label{fig:ptb_loss}
\end{figure}



During the training phase, LEHIBRID (purple) maintained its trajectory below LEHI (green), Adam (red), AdamW (blue), and Sophia-G (orange) (Figure \ref{fig:ptb_loss}). On the testing set, both LEHI (green) and LEHIBRID (purple), with much larger learning rates, outperformed the Adam (red) and AdamW (blue) baselines, which utilized a much smaller learning rate around $ 10^{-5}$. Meanwhile, Sophia-G remained as a close competitor. As in other experiments, this showcases LEHI's ability to leverage larger learning rates effectively. 
Notably, the relative performance differences for the optimizers on the test set closely matched those during validation-based selection.

\subsection{FineWeb-Edu}

FineWeb-Edu is a large-scale language modeling dataset consisting of educational web pages filtered from FineWeb \cite{penedo2024finewebdatasetsdecantingweb}. We used a subset of FineWeb-Edu for next-token prediction and evaluated optimization performance by training a Llama-3.2-1B model \cite{grattafiori2024llama3herdmodels} from scratch. The model has $1.23$B parameters, which provides a larger-scale training setting while remaining computationally tractable. We trained the model using a fixed context length of $1024$ tokens, and a minibatch size of $16$ with $3000$ steps. For each optimizer, the best learning rate was selected based on the testing loss (or log perplexity) score across the final $300$ steps. All other hyperparameters were held fixed: $\beta_1=0.9, \beta_2=0.999$ for Adam and LEHI series, $\epsilon=10^{-7}$ for Adam and AdamW, $\epsilon=10^{-2}$ for LEHI and LEHIBRID, \texttt{weight decay}$=10^{-1}$ for AdamW, and $\beta_1=0.965,\beta_2=0.99,\epsilon=10^{-12},k=10,\rho=0.05$, \texttt{weight decay}$=0.2$ for Sophia-G.

On FineWeb-Edu, LEHI (green) and LEHIBRID (purple) achieved the best testing loss (Figure \ref{fig:fineweb-edu_test_loss}). The two methods exhibited nearly overlapping trajectories, suggesting that the hybrid scheme did not substantially change performance. Adam (red), AdamW (blue), and Sophia-G remained competitive, but converged to higher final losses than LEHI and LEHIBRID. Notably, the selected learning rates for LEHI, LEHIBRID, and Sophia-G were larger than those selected for Adam and AdamW, consistent with the intuition that curvature-aware scaling can support more aggressive steps.

\begin{figure}[t]
    \centering
    \includegraphics[width=0.33\columnwidth]{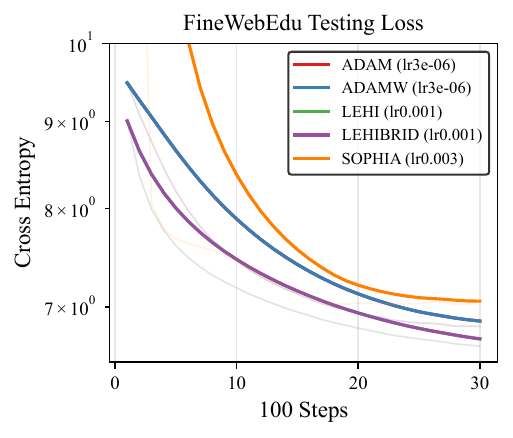}
    \caption{Testing loss (log perplexity) on \textbf{FineWeb-Edu}. Learning rate was selected from a coarse list $\{10^{-1}, 10^{-2},  10^{-3}, 10^{-4},  10^{-5}, 10^{-6},10^{-7}\}$ then refined based on the best testing performance for each optimizer. EMA smoothing ($\alpha=0.2$) is applied. First epoch is omitted to highlight differences. Faded lines show raw curves. Training loss not included since the evaluation was based on steps.}
    \label{fig:fineweb-edu_test_loss}
\end{figure}



\subsection{Sensitivity Analysis on Penn Treebank}

In this subsection, we discuss the learning rate sensitivity of the optimizers on training GPT-2 Small with the Penn Treebank dataset based on the averaged performance on the validation set.

From Figure \ref{fig:ptb_lr_loss_curve}, one can observe both LEHI (green) and LEHIBRID (purple) are more robust with learning rate choices. At larger learning rate as $1 \times 10^{-1}$, LEHI and LEHIBRID retained their performance, while Adam (red), AdamW (blue) and Sophia-G (orange) showed signs of divergence. Additionally, Table \ref{tab:ptb-stability-stripped} shows a simplified summary of the sensitivity analysis results; more extensive results can be found in Table \ref{tab:ptb_stability_summary} in Appendix \ref{app.sub.add}. While LEHI and LEHIBRID were stable across all the grids between $3 \times 10^{-6}$ and $1 \times 10^{-1}$ with all $3$ random seeds, both Adam and AdamW had NaN issues, meaning the results completely diverged. Moreover, for runs without NaN issues, for example, Adam at learning rate $3 \times 10^{-3}$, there were loss function gradient spikes as well, as the largest infinity norm of loss function gradient was around $4.4 \times 10^{32}$.



\begin{figure}[t]
  \centering
  \begin{minipage}[t]{0.38\textwidth}
    \vspace{0pt}
    \centering
    \includegraphics[width=\linewidth]{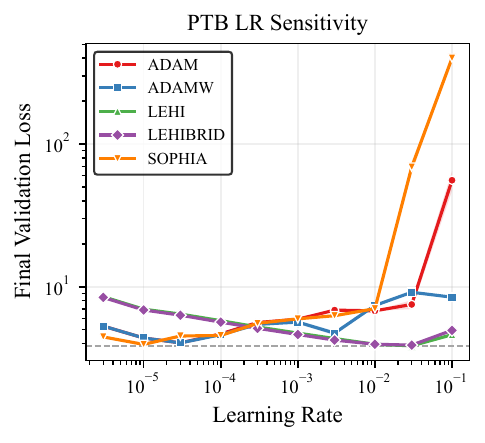}
    \caption{Last 3 epoch average validation loss versus learning rate on PTB. Shaded area presents $2$ times standard deviation across last $3$ epochs. Horizontal dashed line shows lowest average testing loss achieved by LEHI.}
    \label{fig:ptb_lr_loss_curve}
  \end{minipage}\hfill
  \begin{minipage}[t]{0.60\textwidth}
    \vspace{0pt}
    \centering
    \captionof{table}{Stability analysis on PTB across 3 seeds. LR denotes learning rate. $\|\nabla\|_{\infty}$ is the max loss function gradient norm seen. Status symbols: \checkmark (Stable), ! (Noisy), \texttimes (Diverged).}
    \label{tab:ptb-stability-stripped}
    \scriptsize
    \setlength{\tabcolsep}{3pt}
    \begin{tabular}{llccr}
    \toprule
    Opt. & LR & $\|\nabla\|_{\infty}$ & NaNs & Stat. \\
    \midrule
    Adam 
        & 3e-6 & 1.2e-1 & 0 & \checkmark \\
        & 1e-3 & 5.2e-2 & 0 & \checkmark \\
        & 3e-3 & 4.4e+32 & 0 & ! \\
        & 1e-1 & 1.9e+29 & 2 & \texttimes \\
    \midrule
    AdamW 
        & 3e-6 & 1.4e-1 & 0 & \checkmark \\
        & 3e-3 & 6.7e-2 & 0 & \checkmark \\
        & 3e-2 & 2.0e+30 & 1 & \texttimes \\
        & 1e-1 & 2.4e+17 & 0 & ! \\
    \midrule
    \textbf{LEHI} 
        & 3e-6 & 0.14 & 0 & \checkmark \\
        & 1e-2 & 6.4e-2 & 0 & \checkmark \\
        & 1e-1 & \textbf{3.2e-2} & 0 & \checkmark \\
    \midrule
    \textbf{LEHIBRID} 
        & 3e-6 & 0.14 & 0 & \checkmark \\
        & 1e-2 & 6.2e-2 & 0 & \checkmark \\
        & 1e-1 & \textbf{4.7e-2} & 0 & \checkmark \\
    \midrule
    SOPHIA
        & 3e-6 & 0.12 & 0 & \checkmark \\
        & 3e-4 & 9.4 & 0 & \checkmark \\
        & 1e-2 & 3.4e+6 & 0 & ! \\
        & 1e-1 & 1.8e+22 & 0 & ! \\
    \bottomrule
\end{tabular}
  \end{minipage}
\end{figure}

\section{Conclusion}\label{sec.conclusion}

We have proposed, analyzed, and tested a new momentum-based stochastic optimization method for supervised learning that makes use of an auxiliary loss function to approximate second-order derivatives of the original loss function with respect to predictions. The algorithm has computational cost and convergence guarantees that are comparable to Adam, and it along with a hybrid scheme offers consistent performance improvements over Adam and AdamW on a range of test problems. The main challenge of employing the method is the derivation of an appropriate auxiliary loss function. That said, we conjecture that the flexibility offered by the choice of auxiliary loss could offer substantive performance improvements in a wide variety of settings.

\bibliographystyle{plain}
\bibliography{references}

@article{BottCurtNoce2018,
author = {L\'{e}on Bottou and Frank E. Curtis and Jorge Nocedal},
title = {{Optimization Methods for Large-Scale Machine Learning}},
journal = {{SIAM Review}},
volume = {60},
number = {2},
pages = {223--311},
year = {2018}
}

@article{DefoBottBachUsun2022,
  title={A Simple Convergence Proof of Adam and Adagrad},
  author={Alexandre D'efossez and L{\'e}on Bottou and Francis R. Bach and Nicolas Usunier},
  journal={Trans. Mach. Learn. Res.},
  year={2020},
  volume={2022},
  url={https://api.semanticscholar.org/CorpusID:225213299}
}

@article{DuchHazaSing2011,
author = {Duchi, John and Hazan, Elad and Singer, Yoram},
title = {Adaptive Subgradient Methods for Online Learning and Stochastic Optimization},
year = {2011},
publisher = {JMLR.org},
volume = {12},
journal = {J. Mach. Learn. Res.},
pages = {2121--2159}
}

@article{KingBa2014,
  author = {Kingma, Diederik P. and Ba, Jimmy},
  journal = {arXiv preprint arXiv:1412.6980},
  title = {Adam: A method for stochastic optimization},
  year = {2014}
}

@misc{physicochemical_properties_of_protein_tertiary_structure_265,
  author       = {Rana, Prashant},
  title        = {{Physicochemical Properties of Protein Tertiary Structure}},
  year         = {2013},
  howpublished = {UCI Machine Learning Repository},
  note         = {{DOI}: https://doi.org/10.24432/C5QW3H}
}

@article{1571417126193283840,
  author="LeCun, Y.",
  title="THE MNIST DATABASE of handwritten digits",
  journal="http://yann.lecun.com/exdb/mnist/",

  year = {1998},
  URL="https://cir.nii.ac.jp/crid/1571417126193283840"
}

@techreport{krizhevsky2009learning,
  title={Learning multiple layers of features from tiny images},
  author={Krizhevsky, Alex},
  institution={University of Toronto},
  year={2009}
}

@article{marcus-etal-1993-building,
    title = "Building a Large Annotated Corpus of {E}nglish: The {P}enn {T}reebank",
    author = "Marcus, Mitchell P.  and
      Santorini, Beatrice  and
      Marcinkiewicz, Mary Ann",
    editor = "Hirschberg, Julia",
    journal = "Computational Linguistics",
    volume = "19",
    number = "2",
    year = "1993",
    address = "Cambridge, MA",
    publisher = "MIT Press",
    url = "https://aclanthology.org/J93-2004/",
    pages = "313--330"
}

@article{RobbMonr51,
  author  = {Robbins, H. and Monro, S.},
  journal = {The Annals of Mathematical Statistics},
  number  = {3},
  pages   = {400--407},
  title   = {{A Stochastic Approximation Method}},
  volume  = {22},
  year    = {1951}
}

@incollection{RobbSieg71,
  author    = {Robbins, Herbert and Siegmund, David},
  title     = {A convergence theorem for nonnegative almost supermartingales and some applications},
  booktitle = {Optimizing Methods in Statistics},
  publisher = {Academic Press},
  editor    = {Jagdish S. Rustagi},
  year      = {1971}
}

@Misc{TielHint12,
  author =    {Tijmen Tieleman and Geoffrey Hinton},
  title =     {Lecture 6.5. {RMSPROP}: Divide  the  gradient  
                  by  a  running  average  of its recent magnitude.},
  howpublished = {COURSERA: Neural Networks for Machine Learning},
  year =      {2012},
}

@article{loshchilov2017decoupled,
  title={Decoupled weight decay regularization},
  author={Loshchilov, Ilya and Hutter, Frank},
  journal={arXiv preprint arXiv:1711.05101},
  year={2017}
}

@inproceedings{Dozat2016,
  title={Incorporating Nesterov Momentum into Adam},
  author={Timothy Dozat},
  year={2016},
  booktitle={Proceedings of the 4th International Conference on Learning Representations}
}

@misc{liu2021varianceadaptivelearningrate,
      title={On the Variance of the Adaptive Learning Rate and Beyond}, 
      author={Liyuan Liu and Haoming Jiang and Pengcheng He and Weizhu Chen and Xiaodong Liu and Jianfeng Gao and Jiawei Han},
      year={2021},
      eprint={1908.03265},
      archivePrefix={arXiv},
      primaryClass={cs.LG},
      url={https://arxiv.org/abs/1908.03265}, 
}

@article{reddi2019convergence,
  title={On the convergence of adam and beyond},
  author={Reddi, Sashank J and Kale, Satyen and Kumar, Sanjiv},
  journal={arXiv preprint arXiv:1904.09237},
  year={2019}
}

@inproceedings{ICLR2024_06960915,
 author = {Liu, Hong and Li, Zhiyuan and Hall, David and Liang, Percy and Ma, Tengyu},
 booktitle = {International Conference on Learning Representations},
 editor = {B. Kim and Y. Yue and S. Chaudhuri and K. Fragkiadaki and M. Khan and Y. Sun},
 pages = {1621--1650},
 title = {Sophia: A Scalable Stochastic Second-order Optimizer for Language Model Pre-training},
 url = {https://proceedings.iclr.cc/paper_files/paper/2024/file/06960915ba8674c7a898ec0b472b80ff-Paper-Conference.pdf},
 volume = {2024},
 year = {2024}
}

@misc{penedo2024finewebdatasetsdecantingweb,
      title={The FineWeb Datasets: Decanting the Web for the Finest Text Data at Scale}, 
      author={Guilherme Penedo and Hynek Kydlíček and Loubna Ben allal and Anton Lozhkov and Margaret Mitchell and Colin Raffel and Leandro Von Werra and Thomas Wolf},
      year={2024},
      eprint={2406.17557},
      archivePrefix={arXiv},
      primaryClass={cs.CL},
      url={https://arxiv.org/abs/2406.17557}, 
}

@misc{grattafiori2024llama3herdmodels,
      title={The Llama 3 Herd of Models}, 
      author={Aaron Grattafiori et al.},
      year={2024},
      eprint={2407.21783},
      archivePrefix={arXiv},
      primaryClass={cs.AI},
      url={https://arxiv.org/abs/2407.21783}, 
}

\pagebreak
\appendix
\section{Loss Function Derivation}\label{app.auxiliary}

Our aim in this appendix is to provide a few examples of reasonable choices for the vector $v(p(w,x),y) \in \mathbb{R}^q$ such that \eqref{eq.v} holds.  In particular, we provide examples of a loss function $\ell$ such that with corresponding $z$ one obtains that
\begin{equation}\label{eq.diag}
  \text{diag}(v(p(w,x),y) v(p(w,x),y)^T) = \text{diag}(\nabla_{pp}^2 \ell(p(w,x),y)),\ \ \text{where}\ \ v := \nabla_p z
\end{equation}
and $\text{diag}(\cdot)$ denotes the diagonal elements of its symmetric matrix argument.  For simplicity, as in \S\ref{sec.composite}--\S\ref{sec.canonical}, our derivations in this appendix consider a single input-output pair $(x,y)$.  The formulas are easily generalized to sums over a set of input-output pairs.  Observe that since our interest is only in derivatives of $\ell$ and $z$ with respect to $p$, we can ignore the inputs $w$ and $x$ for $p$ and focus simply on the losses $\ell$ and $z$ with respect to the inputs $p \in \mathbb{R}^q$ and $y \in \mathbb{R}^q$.

\subsection{Mean Squared Error (MSE) Loss}

The mean squared error (MSE) loss with respect to a prediction $p$ and true output $y$ is defined by
\begin{equation*}
  \ell(p,y) = \frac{1}{2} \|p - y\|_2^2,\ \ \text{so}\ \ \nabla_p \ell(p,y) = p - y\ \ \text{and}\ \ \nabla_{pp}^2 \ell(p,y) = I.
\end{equation*}
Thus, one can choose $z(p,y) = \mathds{1}^Tp$, where $\mathds{1}$ denotes a vector of ones of appropriate dimension, in order to have $v(p,y) = \nabla_p z(p,y) = \mathds{1}$ such that \eqref{eq.diag} holds.  Notice that $\nabla_p \ell(p,y) = p - y$ and $\nabla_p z(p,y) = \mathds{1}$ can differ significantly, such as when $p \approx y$, as may be likely to occur when minimizing MSE loss.  More generally, for MSE loss that is the average over $N$ input/output pairs, one can similarly choose $z(p,y) = \frac{1}{\sqrt{N}} \mathds{1}^Tp$ since in such a setting one has $\nabla_{pp}^2 \ell(p,y) = \frac{1}{N} I$.

\subsection{Binary Cross-Entropy Loss}

The binary cross-entropy loss with $p \in \mathbb{R}$ and $y \in \mathbb{R}$ is defined by
\begin{equation*}
  \ell(p,y) = -y \log(\sigma(p)) - (1 - y) \log(1 - \sigma(p)),\ \ \text{where}\ \ \sigma(p) = \frac{1}{1 + e^{-p}},
\end{equation*}
from which it follows that
\begin{equation*}
  \frac{d\ell}{dp}(p,y) = \sigma(p) - y\ \ \text{and}\ \ \frac{d^2\ell}{dp^2}(p,y) = \sigma(p) (1 - \sigma(p)).
\end{equation*}
Now consider the auxiliary loss function defined according to
\begin{equation*}
  z(p,y) = \arcsin(\tanh(p/2)),
\end{equation*}
for which one finds that
\begin{align*}
  v(p,y) = \frac{dz}{dp}(p,y)
  &= \left. \frac{d z}{d \tanh(\cdot/2)} \right|_{\tanh(p/2)} \left. \frac{d \tanh(\cdot/2)}{d p} \right|_p \\
  &= \frac{1}{\sqrt{1 - \tanh^2 (p/2)}} \left( \frac12 \left( 1 - \tanh^2 (p/2) \right) \right) \\
  &= \frac12 \sqrt{1 - \tanh^2(p/2)} \\
  &= \frac12 \sqrt{1 - \left( \frac{e^p - 1}{e^p + 1} \right)^2} \\
  &= \frac12 \sqrt{\frac{(e^p + 1)^2}{(e^p + 1)^2} - \frac{e^{2p} - 2e^p + 1}{(e^p + 1)^2}} \\
  &= \frac{e^{p/2}}{e^p + 1} = \frac{1}{e^{p/2} + e^{-p/2}}.
\end{align*}
Consequently, one finds that
\begin{equation*}
  (v(p,y))^2 = \left( \frac{1}{e^{p/2} + e^{-p/2}} \right)^2 = \frac{1}{e^p + 2 + e^{-p}} = \sigma(p) (1 - \sigma(p)) = \frac{d^2\ell}{dp^2}(p,y).
\end{equation*}
More generally, for binary cross-entropy loss over $N$ pairs, one can similarly choose $z(p,y) = \frac{1}{\sqrt{N}} \arcsin(\tanh(p/2))$.

\subsection{Multi-Class Cross-Entropy Loss}

A multi-class cross-entropy loss (in logit form) with $q$ classes, $p \in \mathbb{R}^q$, and $y \in \mathbb{R}^q$ is defined by
\begin{equation*}
  \ell(p,y) = -y^Tp + \log \left( \sum_{i=1}^q e^{p_i} \right),
\end{equation*}
where $p_i$ is the $i$th component of $p$, from which it follows for all $i \in \{1,\dots,q\}$ that
\begin{equation*}
  \frac{d\ell}{dp_i}(p,y) = \sigma(p_i,p) - y\ \ \text{and}\ \ \frac{d^2\ell}{dp_i^2}(p,y) = \sigma(p_i,p)(1 - \sigma(p_i,p)),\ \ \text{where}\ \ \sigma(p_i,p) = \frac{e^{p_i}}{\sum_{j=1}^q e^{p_j}}.
\end{equation*}
Now consider the auxiliary loss function defined according to
\begin{equation*}
  z(p,y) = \sum_{i=1}^q \arcsin(2 \sigma(p_i,p) - 1),
\end{equation*}
for which one finds for all $i \in \{1,\dots,q\}$ that
\begin{align*}
  v_i(p,y) = \frac{dz}{dp_i}(p,y)
  &\approx \left. \frac{dz}{d (2\sigma(\cdot,p) - 1)} \right|_{2\sigma(p_i,p)-1} \left. \frac{d(2\sigma(\cdot,p) - 1)}{dp_i} \right|_{p_i} \\
  &= \frac{1}{\sqrt{1 - (2\sigma(p_i,p) - 1)^2}} (2 \sigma(p_i,p) (1 - \sigma(p_i,p))) \\
  &= \frac{1}{\sqrt{1 - (2\sigma(p_i,p) - 1)^2}} \left( \frac12 \left( 1 - (2\sigma(p_i,p) - 1)^2 \right) \right) \\
  &= \frac12 \sqrt{1 - (2\sigma(p_i,p) - 1)^2}
\end{align*}
Consequently, one finds that
\begin{align*}
  (v_i(p,y))^2 \approx \frac14 (1 - (2\sigma(p_i,p) - 1)^2) &= \frac14 (1 - 4\sigma(p_i,p)^2 + 4\sigma(p_i,p) - 1) \\
  &= \sigma(p_i,p)(1 - \sigma(p_i,p)) = \frac{d^2\ell}{dp_i^2}(p,y).
\end{align*}
More generally, for multi-class cross-entropy loss over $N$ pairs, one can choose $\displaystyle z(p,y) = \frac{1}{\sqrt{N}} \sum_{i=1}^q \arcsin(2\sigma(p_i,p) - 1)$.

\section{Proof of Theorem~\ref{th.lehi}}\label{app.th.lehi}

The $i$th component of the step in iteration $k$ is given by $u_{k,i} := \frac{m_{k,i}}{\sqrt{\epsilon + \tilde{v}_{k,i}}}$.  However, for the purposes of our analysis, let us also define for all $(k,i)$ the related quantity $\hat{u}_{k,i} := \frac{g_{k,i}}{\sqrt{\epsilon + \tilde{v}_{k,i}}}$.  Note that for all $(k,i)$ and $j \in [k-1]$ one has
\begin{equation*}
  \tilde{v}_{k,i} = \sum_{l=1}^k \beta_2^{k-l} \tilde{g}_{l,i}^2 = \beta_2^{j} \tilde{v}_{k-j,i} + \sum_{l = k-j+1}^k \beta_2^{k-l} \tilde{g}_{l,i}^2.
\end{equation*}
In addition, for all $k$ and $j \in [k-1]$, with $\mathbb{E}_{k-j}$ denoting expectation taken up to the beginning of iteration $k-j$, let us also introduce for each $(k,i)$ and $j \in [k-1]$ the related quantity
\begin{equation*}
  \vhat_{k,j,i} = \beta_2^j \tilde{v}_{k-j,i} + \mathbb{E}_{k-j}\left[\sum_{l = k-j+1}^k \beta_2^{k-l} \tilde{g}_{l,i}^2\right].
\end{equation*}

This first lemma will be used in the proof of the theorem to provide a bound for $\|u_k\|_2^2$.  Since it merely involves some tedious calculations with real-number sequences, we refer to proofs in \cite{DefoBottBachUsun2022}.

\begin{lemma}\label{lem.adam.u}
  Let $(\epsilon,\beta_2,\beta_1)$ be given as in Algorithm~\ref{alg.lehi} and let $\{a_k\}$ be a sequence of real numbers.  For any $k \in \mathbb{N}$, with $b_k := \sum_{j=1}^k \beta_2^{k-j} a_j^2$ and $c_k = \sum_{j=1}^k \beta_1^{k-j} a_j$, one has that
  \begin{align*}
    \sum_{j=1}^k \frac{c_j^2}{\epsilon + b_j} &\le \frac{1}{(1-\beta_1)(1-\beta_1/\beta_2)} \left(\log \left(1 + \frac{b_k}{\epsilon}\right) - k \log (\beta_2) \right) \\ \text{and}\ \ 
    \sum_{j=1}^k \frac{a_j^2}{\epsilon + b_j} &\le \log\left(1 + \frac{b_k}{\epsilon} \right) - k \log(\beta_2).
  \end{align*}
\end{lemma}
\begin{proof}
  Please refer to Lemma~A.2 and Lemma~5.2 in~\cite{DefoBottBachUsun2022}.
\end{proof}

Our first critical lemma shows a lower bound for the directional derivative of $f$ at $w_k$ in the direction $u_k$ for all $k \in \mathbb{N}$.

\begin{lemma}\label{lem.adam.descent}
  Under the conditions of Theorem~\ref{th.lehi}, for any $k \in \mathbb{N}$, Algorithm~\ref{alg.lehi} yields
  \begin{align*}
    \mathbb{E}[\nabla f(w_k)^T u_k]
    \geq&\ \frac{1}{2} \sum_{i\in[d]} \sum_{j=0}^{k-1} \beta_1^j \mathbb{E} \left[\frac{\nabla_i f(w_{k-j})^2}{\sqrt{\epsilon + \hat{v}_{k,j+1,i}}} \right] \\
    &\ -\frac{\alpha_k^2 L^2 \sqrt{1 - \beta_1}}{4M} \left(\sum_{l=1}^{k-1} \mathbb{E}[\|u_{k-l}\|_2^2] \sum_{j=l}^{k-1} \beta_1^j \sqrt{j} \right) \\
    & - \frac{3M}{\sqrt{1 - \beta_1}} \left( \sum_{j=0}^{k-1} \left( \frac{\beta_1}{\beta_2} \right)^k \sqrt{j+1} \mathbb{E}[\|\hat{u}_{k-j}\|_2^2] \right).
  \end{align*} 
\end{lemma}
\begin{proof}
  Consider arbitrary $k \in \mathbb{N}$.  One first finds by adding and subtracting terms that
  \begin{align}
    \nabla f(w_k)^T u_k
    &= \sum_{i\in[d]} \nabla_i f(w_k) \frac{m_{k,i}}{\sqrt{\epsilon + \tilde{v}_{k,i}}} = \sum_{i \in [d]} \sum_{j=0}^{k-1} \beta_1^j \nabla_i f(w_k) \frac{g_{k-j,i}}{\sqrt{\epsilon + \tilde{v}_{k,i}}} \nonumber \\
    &= \underbrace{\sum_{i \in [d]} \sum_{j=0}^{k-1} \beta_1^j \nabla_i f(w_{k-j}) \frac{g_{k-j,i}}{\sqrt{\epsilon + \tilde{v}_{k,i}}}}_A \\
    & \qquad + \underbrace{\sum_{i \in [d]} \sum_{j=0}^{k-1} \beta_1^j (\nabla_i f(w_k) - \nabla_i f(w_{k-j}))\frac{g_{k-j,i}}{\sqrt{\epsilon + \tilde{v}_{k,i}}}}_B. \label{eq.adam.descent.1}
  \end{align}
  Let us first bound $B$.  One finds for all $i \in [d]$ and $j \in [k-1]$ that
  \begin{equation}\label{eq.adam.descent.2}
  \begin{aligned}
    \epsilon + \tilde{v}_{k,i} &= \epsilon + \beta_2^j \tilde{v}_{k-j,i} + \sum_{l=k-j+1}^k \beta_2^{k-l} \tilde{g}_{l,i}^2 \geq \epsilon + \beta_2^j \tilde{v}_{k-j,i} \geq \beta_2^j (\epsilon + \tilde{v}_{k-j,i})  \\
    & \qquad \implies \frac{g^2_{k-j,i}}{\epsilon + \tilde{v}_{k,i}} \leq \frac{1}{\beta_2^j} \hat{u}_{k-j,i}^2.
\end{aligned}
  \end{equation}
  At the same time, with Jensen's inequality and the fact that $\{\alpha_k\}$ is nondecreasing, one has for any $j \in [k-1]$ that
  \begin{align}
    \|\nabla f(w_k) - \nabla f(w_{k-j})\|_2^2 &\leq L^2 \|w_k - w_{k-j}\|_2^2 \nonumber \\
    &= L^2 \left\| \sum_{l=1}^{j} \alpha_{k-l} u_{k-l} \right\|_2^2 \nonumber \\
    &\leq L^2 \left( \sum_{l=1}^{j} \alpha_{k-l}^2 \right) \left( \sum_{l=1}^{j} \|u_{k-l}\|_2^2 \right) \nonumber \\
    &\leq L^2 j \alpha_k^2 \left(\sum_{l = 1}^{j} \|u_{k-l}\|_2^2 \right). \label{eq.adam.descent.3}
  \end{align}
  By \eqref{eq.adam.descent.1}, \eqref{eq.adam.descent.2} and \eqref{eq.adam.descent.3}, and since for any pair of real numbers $(\alpha,\beta) \in \mathbb{R} \times \mathbb{R}$ and positive real number $\lambda \in (0,\infty)$ one has that $|\alpha \beta| \leq \frac{\lambda}{2} |\alpha|^2 + \frac{1}{2\lambda} |\beta|^2$, one finds with $\lambda = \frac{\sqrt{1 - \beta_1}}{2M\sqrt{j+1}}$ that
  \begin{align*}
    |B|
    &= \left| \sum_{i \in [d]} \sum_{j=0}^{k-1} \beta_1^j (\nabla_i f(w_k) - \nabla_i f(w_{k-j})) \frac{g_{k-j,i}}{\sqrt{\epsilon + \tilde{v}_{k,i}}} \right| \\
    &\leq \sum_{i \in [d]} \sum_{j=0}^{k-1} \beta_1^j | \nabla_i f(w_k) - \nabla_i f(w_{k-j}) | \left|\frac{g_{k-j,i}}{\sqrt{\epsilon + \tilde{v}_{k,i}}} \right| \\
    &\leq \sum_{i \in [d]} \sum_{j=0}^{k-1} \beta_1^j \left( \frac{\sqrt{1 - \beta_1}}{4M \sqrt{j+1}} \left( \nabla_i f(w_k) - \nabla_i f(w_{k-j}) \right)^2 + \frac{M \sqrt{j+1}}{\sqrt{1-\beta_1}} \frac{g^2_{k-j,i}}{\epsilon + \tilde{v}_{k,i}} \right) \\
    &\leq \sum_{j=0}^{k-1} \beta_1^j \frac{\sqrt{1 - \beta_1}}{4M \sqrt{j+1}} \| \nabla f(w_k) - \nabla f(w_{k-j}) \|_2^2 + \frac{M}{\sqrt{1 - \beta_1}} \sum_{j=0}^{k-1} \frac{\sqrt{j+1} \beta_1^j}{\beta_2^j} \|\hat{u}_{k-j}\|_2^2 \\
    &\leq \frac{\sqrt{1 - \beta_1}}{4M} L^2 \alpha_k^2 \left( \sum_{j=0}^{k-1} \beta_1^j \frac{j}{\sqrt{j+1}} \sum_{l=1}^j \|u_{k-l}\|_2^2 \right) + \frac{M}{\sqrt{1 - \beta_1}}  \sum_{j=0}^{k-1} \left( \frac{\beta_1}{\beta_2} \right)^j \sqrt{j+1} \|\hat{u}_{k-j}\|_2^2 \\
    &\leq \frac{\sqrt{1 - \beta_1}}{4M} L^2  \alpha_k^2 \left( \sum_{l=1}^{k-1} \|u_{k-l}\|_2^2 \sum_{j=l}^{k-1} \beta_1^j \sqrt{j}  \right) + \frac{M}{\sqrt{1 - \beta_1}}  \sum_{j=0}^{k-1} \left( \frac{\beta_1}{\beta_2} \right)^j \sqrt{j+1} \|\hat{u}_{k-j}\|_2^2.
  \end{align*}
  Next, let us bound $A$ in \eqref{eq.adam.descent.1}.  For all $i \in [d]$ and $j \in [k-1]$, one finds that
  \begin{equation*}
    \frac{\nabla_i f(w_{k-j}) g_{k-j,i}}{\sqrt{\epsilon + \tilde{v}_{k,i}}} = \underbrace{\frac{\nabla_i f(w_{k-j}) g_{k-j,i}}{\sqrt{\epsilon + \hat{v}_{k,j+1,i}}}}_{A_1} + \underbrace{\nabla_i f(w_{k-j}) g_{k-j,i} \left( \frac{1}{\sqrt{\epsilon + \tilde{v}_{k,i}}} - \frac{1}{\sqrt{\epsilon + \hat{v}_{k,j+1,i}}}   \right)}_{C}.
  \end{equation*}
  Taking expectation of $A_1$, one finds that
  \begin{equation*}
    \mathbb{E}\left[\frac{\nabla_i f(w_{k-j}) g_{k-j,i}}{\sqrt{\epsilon + \hat{v}_{k,j+1,i}}} \right] = \mathbb{E}\left[ \mathbb{E}_{k-j-1} \left[ \frac{\nabla_i f(w_{k-j}) g_{k-j,i}}{\sqrt{\epsilon + \hat{v}_{k,j+1,i}}} \right]\right] =  \mathbb{E}\left[\frac{\nabla_i f(w_{k-j})^2}{\sqrt{\epsilon + \hat{v}_{k,j+1,i}}} \right].
  \end{equation*}
  As for $C$, an upper bound for the expected value of $|C|$ can be established.  The details are tedious, so for our purposes here let us simply refer to pages 19--20 of \cite{DefoBottBachUsun2022} and say that the result is
  \begin{align*}
    \mathbb{E}[|C|] \leq \frac{1}{2} \mathbb{E} \left[\frac{\nabla_i f(w_{k-j})^2}{\sqrt{\epsilon + \hat{v}_{k,j+1,i}}}\right] + \frac{2M}{\sqrt{1 - \beta_1} \beta_2^j} \sqrt{j+1} \mathbb{E}\left[\frac{g_{k-j,i}^2}{\epsilon + \vtilde_{k-j,i}}\right].
  \end{align*}
  Hence, one obtains that 
  \begin{align*}
    A
    &= \sum_{i \in [d]} \sum_{j=0}^{k-1} \beta_1^j \left( \frac{\nabla_i f(w_{k-j}) g_{k-j,i}}{\sqrt{\epsilon + \hat{v}_{k,j+1,i}}} + \nabla_i f(w_{k-j}) g_{k-j,i} \left( \frac{1}{\sqrt{\epsilon + \tilde{v}_{k,i}}} - \frac{1}{\sqrt{\epsilon + \hat{v}_{k,j+1,i}}} \right)\right) \\
    &\geq \sum_{i \in [d]} \sum_{j=0}^{k-1} \beta_1^j \left( \frac{\nabla_i f(w_{k-j}) g_{k-j,i}}{\sqrt{\epsilon + \hat{v}_{k,j+1,i}}} - \left|\nabla_i f(w_{k-j}) g_{k-j,i} \left( \frac{1}{\sqrt{\epsilon + \tilde{v}_{k,i}}} - \frac{1}{\sqrt{\epsilon + \hat{v}_{k,j+1,i}}} \right) \right| \right) \\
  \end{align*}
  satisfies the lower bound in expectation given by
  \begin{align*}
    \mathbb{E}[A]
    &\geq \sum_{i \in [d]} \sum_{j=0}^{k-1} \beta_1^j \left( \mathbb{E} \left[ \frac{\nabla_i f(w_{k-j})^2}{\sqrt{\epsilon + \hat{v}_{k,j+1,i}}} \right] - \frac{1}{2} \mathbb{E} \left[ \frac{\nabla_i f(w_{k-j})^2}{\sqrt{\epsilon + \hat{v}_{k,j+1,i}}} \right] - \frac{2M}{\sqrt{1 - \beta_1} \beta_2^j} \sqrt{j+1} \mathbb{E} \left[\frac{g_{k-j,i}^2}{\epsilon + \tilde{v}_{k-j,i}} \right] \right) \\
    &= \frac{1}{2} \sum_{i \in [d]} \sum_{j=0}^{k-1} \beta_1^j \mathbb{E} \left[ \frac{\nabla_i f(w_{k-j})^2}{\sqrt{\epsilon + \hat{v}_{k,j+1,i}}} \right] - \frac{2M}{\sqrt{1 - \beta_1}}   \sum_{j=0}^{k-1} \left( \frac{\beta_1}{\beta_2} \right)^j \sqrt{j+1} \mathbb{E} \left[\|\hat{u}_{k-j}\|_2^2 \right].
  \end{align*}
  Thus,
  \begin{align*}
    \mathbb{E} \left[\nabla f(w_k)^T u_k \right] \geq&\ \mathbb{E}[A] - \mathbb{E}[|B|] \\
    \geq&\ \frac{1}{2} \sum_{i \in [d]} \sum_{j=0}^{k-1} \beta_1^j \mathbb{E} \left[ \frac{\nabla_i f(w_{k-j})^2}{\sqrt{\epsilon + \hat{v}_{k,j+1,i}}} \right] - \frac{2M}{\sqrt{1-\beta_1}} \sum_{j=0}^{k-1} \left( \frac{\beta_1}{\beta_2} \right)^j \sqrt{j+1} \mathbb{E}\left[\|\hat{u}_{k-j}\|_2^2 \right] \\
    &\ - \mathbb{E} \left[ \frac{\sqrt{1 - \beta_1}}{4M} L^2 \alpha_k^2  \left( \sum_{l=1}^{k-1} \|u_{k-l}\|_2^2 \sum_{j=l}^{k-1} \beta_1^j \sqrt{j} \right) \right] \\
    & - \mathbb{E} \left[ \frac{M}{\sqrt{1 - \beta_1}} \sum_{j=0}^{k-1} \left( \frac{\beta_1}{\beta_2} \right)^j \sqrt{j+1} \|\hat{u}_{k-j}\|_2^2 \right],
  \end{align*}
  which completes the proof.
\end{proof}

The only other thing that we need for the proof of the theorem is the following lemma about series.

\begin{lemma}\label{lem:adam.geometric.sequence}
  For any $k \in \mathbb{N}$ and $\beta \in (0,1)$, it follows that
  \begin{equation*}
    \sum_{j=0}^{k-1} \beta^j \sqrt{j+1} \leq \frac{2}{(1-\beta)^{3/2}}\ \ \text{and}\ \ \sum_{j=0}^{k-1} \beta^j \sqrt{j}(j+1) \leq \frac{4\beta}{(1-\beta)^{5/2}},
  \end{equation*}
\end{lemma}
\begin{proof}
  Please refer to Lemmas~A.3 and A.4 from \cite{DefoBottBachUsun2022}.
\end{proof}

Now we are ready to prove the theorem.

\begin{proof}[Proof of Theorem \ref{th.lehi}]
  By Lipschitz continuity of the gradient of $f$, one has \cite{BottCurtNoce2018} that
  \begin{equation}\label{eq.stars}
    f(w_{k+1}) \leq f(w_k) - \alpha_k \nabla f(w_k)^T u_k + \frac{1}{2} \alpha_k^2 L \|u_k\|_2^2.
  \end{equation}
  Under the conditions of the theorem, one finds that
  \begin{equation*}
    \sqrt{\epsilon + \hat{v}_{k,j+1,i}} \leq \sqrt{\epsilon + \sum_{j=1}^k \beta_2^{k-j} (M^2 - \epsilon)} \leq \sqrt{\sum_{j=1}^k \beta_2^{k-j} (M^2 - \epsilon + \epsilon)} = M \sqrt{\sum_{j=0}^{k-1} \beta_2^j} = M\sqrt{\frac{1 - \beta_2^k}{1-\beta_2}} = \frac{M \alpha_k}{\alpha (1 - \beta_1)}.
  \end{equation*}
  Hence, taking total expectation in \eqref{eq.stars} and using Lemma~\ref{lem.adam.descent}, one finds that
  \begin{align*}
    \mathbb{E}[f(w_{k+1})]
    \leq&\ \mathbb{E}[f(w_k)] - \alpha_k \mathbb{E}\left[\nabla f(w_k)^T u_k\right] + \frac{1}{2} \alpha_k^2L \mathbb{E}[\|u_k\|_2^2] \\
    \leq&\ \mathbb{E}[f(w_k)] - \frac{1}{2} \alpha_k \sum_{i\in[d]} \sum_{j=0}^{k-1} \beta_1^j \mathbb{E}\left[\frac{\nabla_i f(w_{k-j})^2}{\sqrt{\epsilon + \hat{v}_{k,j+1,i}}} \right] + \frac{1}{2} \alpha_k^2 L \mathbb{E}[\|u_k\|_2^2] \\
    &\ + \frac{\alpha_k^3 L^2 \sqrt{1-\beta_1}}{4M} \left( \sum_{l=1}^{k-1} \mathbb{E}[\|u_{k-l}\|_2^2] \sum_{j=l}^{k-1} \beta_1^j \sqrt{j} \right) + \frac{3M\alpha_k}{\sqrt{1-\beta_1}} \left(\sum_{j=0}^{k-1} \left( \frac{\beta_1}{\beta_2} \right)^k \sqrt{j+1} \mathbb{E}[\|\hat{u}_{k-j}\|_2^2] \right) \\
    \leq&\ \mathbb{E}[f(w_k)] - \frac{\alpha (1 - \beta_1)}{2M} \left( \sum_{j=0}^{k-1} \beta_1^j \mathbb{E}[\|\nabla f(w_{k-j})\|_2^2 ]\right) + \frac{1}{2} \alpha_k^2L \mathbb{E}[\|u_k\|_2^2] \\
    &\ + \frac{\alpha_k^3 L^2 \sqrt{1-\beta_1}}{4M} \left( \sum_{l=1}^{k-1} \mathbb{E}[\|u_{k-l}\|_2^2] \sum_{j=l}^{k-1} \beta_1^j \sqrt{j}\right) + \frac{3M\alpha_k}{\sqrt{1 - \beta_1}} \left(\sum_{j=0}^{k-1} \left( \frac{\beta_1}{\beta_2} \right)^k \sqrt{j+1} \mathbb{E}[ \|\hat{u}_{k-j}\|_2^2 ] \right). \\
  \end{align*}
  Summing over $k \in [K]$ and using the fact that $\{\alpha_k\}$ is nondecreasing, one has
  \begin{align*}
    \underbrace{\frac{\alpha (1 - \beta_1)}{2M} \sum_{k\in[K]} \sum_{j=0}^{k-1} \beta_1^j \mathbb{E}\left[\|\nabla f(w_{k-j})\|_2^2 \right]}_A
    \leq&\ f(w_1) - \mathbb{E}[f(w_{K+1})] \\
    &\ + \underbrace{\frac{\alpha_K^2 L}{2} \mathbb{E} \left[ \sum_{k\in[K]} \|u_k\|_2^2 \right]}_B \\ 
    &\ + \underbrace{ \frac{\alpha_K^3 L^2 \sqrt{1 - \beta_1}}{4M} \mathbb{E}\left[ \sum_{k\in[K]} \left( \sum_{l=1}^{k-1} \|u_{k-l}\|_2^2 \sum_{j=l}^{k-1} \beta_1^j \sqrt{j} \right) \right]}_C \\
    &\ + \underbrace{\frac{3M\alpha_K}{\sqrt{1 - \beta_1}} \mathbb{E} \left[ \sum_{k\in[K]} \left( \sum_{j=0}^{k-1} \left( \frac{\beta_1}{\beta_2} \right)^k \sqrt{j+1} \|\hat{u}_{k-j}\|_2^2 \right) \right]}_D.
  \end{align*}
  For the $C$ term, with details on page 22 in \cite{DefoBottBachUsun2022}, one finds with Lemma~\ref{lem:adam.geometric.sequence} that
  \begin{equation*}
    C = \frac{\alpha_K^3 L^2 \sqrt{1-\beta_1}}{4M} \mathbb{E} \left[\sum_{k\in[K]}  \|u_k\|_2^2 \sum_{j=0}^{K-1} \beta_1^j \sqrt{j} (j+1) \right] \leq \frac{\alpha_K^3 L^2}{M} \frac{\beta_1}{(1 - \beta_1)^2} \mathbb{E}\left[ \sum_{k\in[K]} \|u_k\|_2^2 \right].
  \end{equation*}
  For the $D$ term, with details on page 23 in \cite{DefoBottBachUsun2022}, one finds with Lemma~\ref{lem:adam.geometric.sequence} that
  \begin{equation*}
    D = \frac{3 M \alpha_K}{\sqrt{1 - \beta_1}} \mathbb{E}\left[ \sum_{k\in[K]} \|\hat{u}_k\|_2^2 \sum_{j=k}^K \left( \frac{\beta_1}{\beta_2} \right)^{j-k} \sqrt{1+j-k} \right] \leq \frac{6 M \alpha_K}{\sqrt{1 - \beta_1}} \frac{1}{(1 - \frac{\beta_1}{\beta_2})^{3/2}} \mathbb{E} \left[ \sum_{k\in[K]} \|\hat{u}_k\|_2^2 \right].
  \end{equation*}
  Now with respect to $\sum_{k\in[K]} \|u_k\|_2^2$ and $\sum_{k\in[K]} \|\hat{u}_k\|_2^2$, one finds by Lemma~\ref{lem.adam.u} and $\tilde{v}_{K,i} \leq \frac{M^2}{1 - \beta_2}$ that
  \begin{align*}
    \sum_{k\in[K]} \|u_k\|_2^2
    &= \sum_{i \in [d]} \sum_{k\in[K]} u_{k,i}^2 \\
    &\leq \sum_{i \in [d]} \frac{1}{(1 - \beta_1)(1 - \frac{\beta_1}{\beta_2})} \left( \log \left(1 + \frac{\tilde{v}_{K,i}}{\epsilon} \right) - K \log(\beta_2) \right) \\
    &\leq \sum_{i \in [d]} \frac{1}{(1 - \beta_1)(1 - \frac{\beta_1}{\beta_2})} \left( \log \left(1 + \frac{M^2}{(1 - \beta_2) \epsilon} \right) - K \log(\beta_2) \right) \\
  \end{align*}
  and
  \begin{align*}
    \sum_{k\in[K]} \|\hat{u}_k\|_2^2
    &= \sum_{i \in [d]} \sum_{k\in[K]} \hat{u}_{k,i}^2 \\
    &\leq \sum_{i \in [d]} \left( \log \left(1 + \frac{\tilde{v}_{K,i}}{\epsilon} \right) - K \log(\beta_2) \right) \\
    &\leq \sum_{i \in [d]} \left( \log \left(1 + \frac{M^2}{(1 - \beta_2) \epsilon} \right) - K \log (\beta_2) \right).
  \end{align*}
  For the $A$ term, one finds that
  \begin{align*}
    A &= \frac{\alpha (1 - \beta_1)}{2M} \sum_{k\in[K]} \sum_{j=0}^{k-1} \beta_1^j \mathbb{E}[\|\nabla f(w_{k-j})\|_2^2] = \frac{\alpha}{2M} \sum_{k\in[K]} (1 - \beta_1^{K-k+1}) \mathbb{E}[\|\nabla f(w_k)\|_2^2 ].
  \end{align*}
  Thus, with the fact thats that $\alpha_K \leq \alpha \frac{1 - \beta_1}{\sqrt{1 - \beta_2}}$ and $\{\alpha_k\}$ is nondecreasing, the telescoping sum gives
  \begin{align}
    &\ \frac{\alpha}{2M} \sum_{k\in[K]} (1 - \beta_1^{K-k+1}) \mathbb{E}[ \|\nabla f(w_k)\|_2^2 ] \nonumber \\
    \leq&\ f(w_1) - f_{\inf} \nonumber \\
    &\ + \left( \frac{\alpha^2 (1 - \beta_1) L}{2(1 - \beta_2)} + \frac{\alpha^3 L^2 \beta_1}{M (1 - \beta_2)^{3/2}} \right) \frac{d}{(1 - \frac{\beta_1}{\beta_2})} \left( \log \left(1 + \frac{M^2}{(1 - \beta_2) \epsilon} \right) - K \log (\beta_2) \right) \nonumber \\
    &\ + \frac{6 \alpha \sqrt{1 - \beta_1} M}{\sqrt{1 - \beta_2}} \frac{d} {(1 - \frac{\beta_1}{\beta_2})^{3/2}}  \left( \log \left(1 + \frac{M^2}{(1 - \beta_2) \epsilon} \right) - K \log (\beta_2) \right). \label{stargazing}
  \end{align}
  Now one can observe that
  \begin{align*}
    \sum_{k=1}^K (1 - \beta_1^{K-k+1}) = K - \sum_{k=1}^K \beta_1^k \geq K - \frac{\beta_1}{1 - \beta_1} =: \tilde{K}
  \end{align*}
  in order to find from the definition of $R_K$ that
  \begin{align*}
    \mathbb{E}[ \|\nabla f(w_{R_K})\|_2^2 ] = \sum_{k=1}^K \mathbb{P}[R_K = k] \mathbb{E} [ \|\nabla f(w_k) \|_2^2 ]
    &= \frac{ \sum_{k \in [K]} \mathbb{E}[ \|\nabla f(w_k)\|_2^2 ] (1 - \beta_1^{K-k+1})}{\sum_{k\in[K]} (1 - \beta_1^{K-k+1})} \\
    &\leq \frac{ \sum_{k\in[K]} \mathbb{E}[ \|\nabla f(w_k)\|_2^2 ] (1 - \beta_1^{K-k+1})}{\tilde{K}}.
  \end{align*}
  The desired conclusion follows from this inequality along with \eqref{stargazing}.
\end{proof}

\section{Additional Experiment Details}
We implemented and tested our algorithm on the following settings (Table \ref{tab:experiment_summary}).











\begin{table}[t]
\centering
\caption{Experiment summary}
\label{tab:experiment_summary}
\begin{tabular}{lll}
\toprule
\textbf{Neural network structure} & \textbf{Dataset} & \textbf{Loss function} \\
\midrule
MLP (1 hidden layer, 100 nodes) & UCI Protein & MSE \\
MLP (1 hidden layer, 50 nodes) & MNIST & Cross-Entropy \\
ResNet-18 & CIFAR-10 & Cross-Entropy \\
ResNet-18 & CIFAR-100 & Cross-Entropy \\
GPT-2 Small & Penn Treebank & Cross-Entropy \\
Llama-3.2-1B & FineWeb-Edu & Cross-Entropy \\
\bottomrule
\end{tabular}
\end{table}

As an implementation-level reference, Table~\ref{tab:runtime_stats} reports per-step runtime and peak GPU memory from one representative FineWeb-Edu run for each optimizer. Runtime statistics are computed from average step times over 10-step logging windows in one representative run. These numbers are intended to provide a coarse indication rather than a full systems benchmark. The LEHI variants incur additional runtime  while their peak memory usage remains comparable to Adam-type baselines and below that of Sophia-G in this run.

\textbf{The runtime overhead of LEHI and LEHIBRID is mainly due to the current non-optimized PyTorch implementation.} Mathematically, the original and auxiliary gradients share the same model-side Jacobian:
\begin{equation*}
    \nabla_w f = \nabla_w p \, \nabla_p \ell,
    \qquad
    \nabla_w \tilde f = \nabla_w p \, \nabla_p z.
\end{equation*}
Thus, both gradients involve the same derivative of the prediction function $p$ with respect to the parameters $w$ (i.e., $\nabla_w p$), and differ only in the loss-side vector propagated through this Jacobian. In principle, this shared structure suggests that the two gradient computations need not be entirely independent. However, standard PyTorch backpropagation is implemented through vector-Jacobian products rather than by explicitly forming and storing $\nabla_w p$. As a result, changing the vector being backpropagated from $\nabla_p \ell$ to $\nabla_p z$ typically requires a second backward pass through the same computation graph. The timings in Table~\ref{tab:runtime_stats} should therefore be interpreted as reflecting a straightforward PyTorch implementation rather than an optimized fused implementation.

Moreover, similar to amortized second-order strategies such as Sophia-G, one could compute and update the auxiliary-gradient information less frequently, while reusing the resulting scaling information across multiple optimization steps. We expect that such an amortized variant may reduce runtime substantially while retaining much of the benefit of the Hessian-imitation scaling.


\begin{table}[t]
\centering
\small
\caption{Runtime and peak memory profile on FineWeb-Edu}
\label{tab:runtime_stats}
\begin{threeparttable}
\begin{tabular}{lrrr}
\toprule
\textbf{Optimizer} & \textbf{Max 10-step avg.} & \textbf{Min 10-step avg.} & \textbf{Peak memory} \\
 & \textbf{(ms/step)} & \textbf{(ms/step)} & \textbf{(MB)} \\
\midrule
Adam     & 1062.86 & 1053.58 & 53103.1 \\
AdamW    & 1069.20 & 1060.50 & 53103.1 \\
LEHI     & 2256.76\tnote{*} & 2221.09\tnote{*} & 58967.0 \\
LEHIBRID & 1659.60\tnote{*} & 1648.98\tnote{*} & 58967.0 \\
Sophia-G & 1173.18 & 1157.91 & 61119.6 \\
\bottomrule
\end{tabular}
\begin{tablenotes}[flushleft]
\footnotesize
\item[*] Please refer to our comments following the table.
\end{tablenotes}
\end{threeparttable}
\end{table}

\subsection{Additional Experiment Results}
\subsubsection{CIFAR-10} \label{app.sub.cifar10}
CIFAR-10 is a widely used image classification benchmark consisting of $32 \times 32$ color images drawn from $10$ categories \cite{krizhevsky2009learning}. We trained a ResNet-18 network using a minibatch size of $256$ with $72$ epochs. For each optimizer, the best learning rate was selected based on the testing accuracy score across the final $7$ epochs. All other hyperparameters were held fixed: $\beta_1=0.9, \beta_2=0.999$ for all optimizers, $\epsilon=10^{-7}$ for Adam and AdamW, $\epsilon=10^{-2}$ for LEHI and LEHIBRID, and \texttt{weight decay}$=10^{-2}$.

The accuracy trajectories for the optimizers are shown in Figure~\ref{fig:cifar10_acc}. For testing accuracies, while initially all optimizers behaved similarly, the zoomed-in subplot from epoch $60$ to $72$ highlights that LEHI (green) and LEHIBRID (purple) maintained a consistent lead in the final epochs compared to Adam (red) and AdamW (blue). In contrast, the learning trajectories for ADAM and ADAMW showed a relative deceleration after 50 epochs, resulting in a relatively lower final testing accuracy.

\begin{figure}[t]
  \centering
  \begin{tabular}{cc}
    \includegraphics[width=0.35\textwidth]{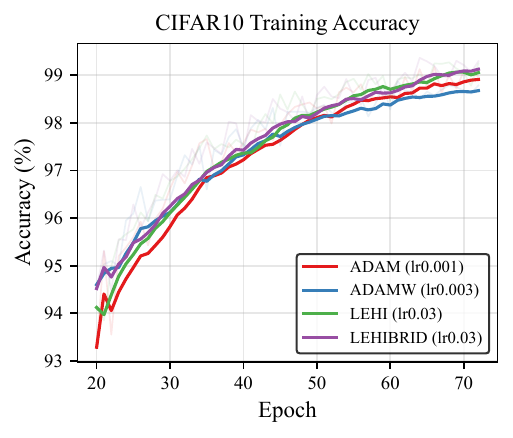} &
    \includegraphics[width=0.35\textwidth]{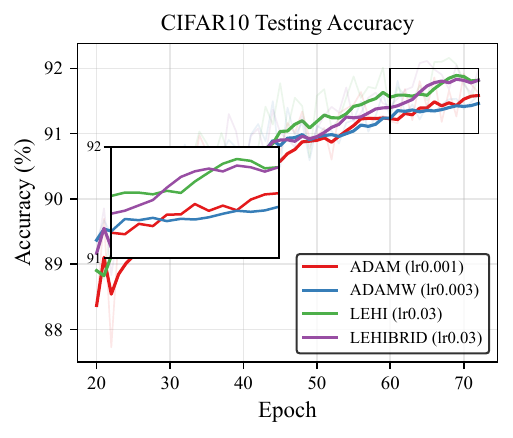} \\
    (a) Training accuracy & (b) Testing accuracy
  \end{tabular}
  \caption{\textbf{CIFAR10} accuracy curves. Learning rate was selected from $\{10^{-1}, 3\times10^{-2},10^{-2}, 3\times10^{-3}, 10^{-3}, 3 \times 10^{-4}, 10^{-4}\}$ based on the best testing performance. EMA smoothing ($\alpha=0.2$) is applied. First $20$ epochs are omitted to highlight differences. Faded lines show raw curves.}
  \label{fig:cifar10_acc}
\end{figure}



\subsection{Experimental Setup}
Our experiments except the FineWeb-Edu were conducted on the Google Colab platform utilizing L4 GPUs. The FineWeb-Edu tasks were conducted on the Google Colab platform utilizing A100 GPUs. For the regression task on UCI Protein, we implemented a randomized 80/20 train-test split and standardized both features and targets using training set statistics. Image classification benchmarks followed official dataset partitions: MNIST used the standard 60,000/10,000 split with single-channel normalization, while CIFAR-10 and CIFAR-100 utilized the official 50,000/10,000 splits. For the CIFAR experiments, training data was preprocessed with random cropping (4-pixel padding) and horizontal flipping, followed by dataset-specific normalization. For language modeling on Penn Treebank, we used the official pre-split text files (train, validation, and test) and the GPT2TokenizerFast to tokenize data into sequences with a block size of 128 tokens, where targets were shifted by one token for autoregressive prediction. For the FineWeb-Edu experiments, we used a subset containing $2$M training tokens and $200$K testing tokens. The subset was used to make the Llama-3.2-1B from-scratch training experiment computationally tractable while preserving the next-token prediction structure of the full FineWeb-Edu benchmark.

\subsection{Learning Rate Tuning and Statistics} \label{app.sub.tuning}
We determined the optimal learning rate for each optimizer and task via a log-scale grid search over a comprehensive range of values. This initial tuning phase was conducted using a fixed random seed ($0$). Based on these search results, we identified the best-performing learning rate candidates for each configuration. To ensure statistical significance, we subsequently evaluated these candidates using two additional random seeds ($1$ and $2$), the results of which are reported in Section \ref{sec.numerical}. 

For the Penn Treebank task, we followed the standard three-split protocol (train/validation/test), using the validation set for hyperparameter selection and reporting final results on the held-out testing set. For the FineWeb-Edu task and smaller-scale benchmarks (MNIST, UCI, CIFAR), we followed the common experimental setup where hyperparameters are selected based on final test split performance. To ensure a fair comparison, we apply an identical tuning setting to all evaluated optimizers; thus, while the absolute performance may reflect testing-set tuning, the relative improvements between methods remain valid.

In this subsection, we include the results we obtained from the first-round learning rate search.

\subsubsection{UCI Protein}
For UCI Protein, we did a log-scale grid search over the set $\{1 \times 10^{-1}, 3 \times 10^{-2}, 1 \times 10^{-2}, 3 \times 10^{-3}, 1 \times 10^{-3}, 3 \times 10^{-4}, 1 \times 10^{-4}\}$. Table \ref{tab:protein_lr_loss} reports the mean testing loss and the associated stability metric ($2 \times$ standard deviation) calculated over the final $10$ epochs of training. The final risk-adjusted selection scores are reported in Table \ref{tab:protein_loss_selection}. We selected learning rates $\{1 \times 10^{-1}, 3 \times 10^{-2}, 1 \times 10^{-2},3 \times 10^{-3}\}$ for LEHI and LEHIBRID, and $\{1 \times 10^{-2}, 3 \times 10^{-3}, 1 \times 10^{-3}, 3 \times 10^{-3}\}$ for Adam and AdamW to conduct further evaluations across two additional random seeds ($1$ and $2$).

\begin{table}[h!]
\centering
\small
\caption{UCI Protein mean testing loss $\pm \;2 \times \text{standard deviation}$ across last $10$ epochs}
\label{tab:protein_lr_loss}
\begin{tabular}{lcccc}
\toprule
\textbf{LR} & \textbf{LEHI} & \textbf{LEHIBRID} & \textbf{ADAM} & \textbf{ADAMW} \\
\midrule
$1 \times 10^{-4}$ & $0.3500 \pm 0.0009$ & $0.3402 \pm 0.0008$ & $0.2727 \pm 0.0007$ & $0.2738 \pm 0.0007$ \\
$3 \times 10^{-4}$ & $0.3251 \pm 0.0005$ & $0.3189 \pm 0.0006$ & $0.2582 \pm 0.0016$ & $0.2599 \pm 0.0015$ \\
$1 \times 10^{-3}$ & $0.3028 \pm 0.0006$ & $0.2965 \pm 0.0006$ & $\mathbf{0.2475 \pm 0.0058}$ & $\mathbf{0.2531 \pm 0.0058}$ \\
$3 \times 10^{-3}$ & $0.2818 \pm 0.0006$ & $0.2752 \pm 0.0006$ & $0.2482 \pm 0.0101$ & $0.2580 \pm 0.0191$ \\
$1 \times 10^{-2}$ & $0.2620 \pm 0.0031$ & $0.2596 \pm 0.0034$ & $0.2583 \pm 0.0123$ & $0.2757 \pm 0.0084$ \\
$3 \times 10^{-2}$ & $0.2535 \pm 0.0070$ & $0.2480 \pm 0.0072$ & $0.2894 \pm 0.0133$ & $0.3038 \pm 0.0139$ \\
$1 \times 10^{-1}$ & $\mathbf{0.2426 \pm 0.0089}$ & $\mathbf{0.2407 \pm 0.0097}$ & $0.3449 \pm 0.0172$ & $0.3594 \pm 0.0215$ \\
\bottomrule
\end{tabular}
\end{table}

\begin{table}[h!]
\centering
\caption{UCI Protein testing loss risk-adjusted selection scores ($c=2.0$)}
\label{tab:protein_loss_selection}
\begin{tabular}{lcccc}
\toprule
\textbf{LR} & \textbf{LEHI} & \textbf{LEHIBRID} & \textbf{ADAM} & \textbf{ADAMW} \\
\midrule
$1 \times 10^{-4}$ & 0.3510 & 0.3410 & 0.2734 & 0.2745 \\
$3 \times 10^{-4}$ & 0.3256 & 0.3195 & 0.2598 & 0.2614 \\
$1 \times 10^{-3}$ & 0.3033 & 0.2971 & \textbf{0.2534} & \textbf{0.2589} \\
$3 \times 10^{-3}$ & 0.2824 & 0.2759 & 0.2584 & 0.2770 \\
$1 \times 10^{-2}$ & 0.2651 & 0.2630 & 0.2706 & 0.2841 \\
$3 \times 10^{-2}$ & 0.2605 & 0.2553 & 0.3027 & 0.3177 \\
$1 \times 10^{-1}$ & \textbf{0.2515} & \textbf{0.2504} & 0.3621 & 0.3810 \\
\bottomrule
\end{tabular}
\end{table}

\subsubsection{MNIST}
For MNIST, we did a log-scale grid search over the set $\{1 \times 10^{-1}, 3 \times 10^{-2}, 1 \times 10^{-2}, 3 \times 10^{-3}, 1 \times 10^{-3}, 3 \times 10^{-4}, 1 \times 10^{-4}\}$. 
Table \ref{tab:mnist_lr_acc} reports the mean testing accuracy and the associated stability metric ($2 \times$ standard deviation) calculated over the final $3$ epochs of training. The final risk-adjusted selection scores are reported in Table \ref{tab:mnist_acc_selection}. We selected learning rates $\{1 \times 10^{-2}, 3 \times 10^{-3}, 1 \times 10^{-3},3 \times 10^{-4}\}$ for LEHI and LEHIBRID, and $\{3 \times 10^{-3}, 1 \times 10^{-3}, 3 \times 10^{-4}, 1 \times 10^{-4}\}$ for Adam and AdamW to conduct further evaluations across two additional random seeds ($1$ and $2$).

\begin{table}[h!]
\centering
\small
\caption{MNIST mean testing accuracy $\pm \; 2 \times \text{standard deviation}$ across last $3$ epochs}
\label{tab:mnist_lr_acc}
\begin{tabular}{lcccc}
\toprule
\textbf{LR} & \textbf{LEHI (\%)} & \textbf{LEHIBRID (\%)} & \textbf{ADAM (\%)} & \textbf{ADAMW (\%)} \\
\midrule
$1 \times 10^{-4}$ & $91.007 \pm 0.186$ & $91.763 \pm 0.114$ & $96.203 \pm 0.160$ & $96.213 \pm 0.180$ \\
$3 \times 10^{-4}$ & $93.847 \pm 0.167$ & $94.593 \pm 0.121$ & $\mathbf{97.347 \pm 0.160}$ & $\mathbf{97.250 \pm 0.183}$ \\
$1 \times 10^{-3}$ & $96.423 \pm 0.155$ & $96.867 \pm 0.090$ & $97.127 \pm 0.301$ & $97.203 \pm 0.375$ \\
$3 \times 10^{-3}$ & $\mathbf{97.293 \pm 0.115}$ & $\mathbf{97.283 \pm 0.163}$ & $96.933 \pm 0.566$ & $96.950 \pm 0.242$ \\
$1 \times 10^{-2}$ & $96.240 \pm 0.481$ & $96.260 \pm 0.180$ & $96.177 \pm 0.691$ & $95.983 \pm 0.266$ \\
$3 \times 10^{-2}$ & $93.920 \pm 0.883$ & $94.343 \pm 0.492$ & $91.450 \pm 1.753$ & $92.450 \pm 1.280$ \\
$1 \times 10^{-1}$ & $88.983 \pm 0.464$ & $82.040 \pm 2.560$ & $31.287 \pm 7.072$ & $10.023 \pm 2.461$ \\
\bottomrule
\end{tabular}
\end{table}

\begin{table}[h!]
\centering
\caption{MNIST testing accuracy risk-adjusted selection scores ($c=2.0$)}
\label{tab:mnist_acc_selection}
\begin{tabular}{lcccc}
\toprule
\textbf{LR} & \textbf{LEHI (\%)} & \textbf{LEHIBRID (\%)} & \textbf{ADAM (\%)} & \textbf{ADAMW (\%)} \\
\midrule
$1 \times 10^{-4}$ & 90.821 & 91.650 & 96.043 & 96.033 \\
$3 \times 10^{-4}$ & 93.680 & 94.473 & \textbf{97.186} & \textbf{97.067} \\
$1 \times 10^{-3}$ & 96.268 & 96.776 & 96.826 & 96.828 \\
$3 \times 10^{-3}$ & \textbf{97.178} & \textbf{97.120} & 96.367 & 96.708 \\
$1 \times 10^{-2}$ & 95.759 & 96.080 & 95.486 & 95.717 \\
$3 \times 10^{-2}$ & 93.037 & 93.852 & 89.697 & 91.170 \\
$1 \times 10^{-1}$ & 88.520 & 79.480 & 24.215 & 7.563 \\
\bottomrule
\end{tabular}
\end{table}

\subsubsection{CIFAR-10}
For CIFAR-10, we did a log-scale grid search over the set $\{1 \times 10^{-1}, 3 \times 10^{-2}, 1 \times 10^{-2}, 3 \times 10^{-3}, 1 \times 10^{-3}, 3 \times 10^{-4}, 1 \times 10^{-4}\}$. 
Table \ref{tab:cifar10_lr_acc} reports the mean testing accuracy and the associated stability metric ($2 \times$ standard deviation) calculated over the final $10$ epochs of training. The final risk-adjusted selection scores are reported in Table \ref{tab:cifar10_acc_selection}. We selected learning rates $\{3 \times 10^{-2}, 1 \times 10^{-2}\}$ for LEHI and LEHIBRID, and $\{3 \times 10^{-3}, 1 \times 10^{-3}\}$ for Adam and AdamW to conduct further evaluations across two additional random seeds ($1$ and $2$).

\begin{table}[h!]
\centering
\small
\caption{CIFAR-10 mean testing accuracy $\pm \; 2 \times \text{standard deviation}$ across last $8$ epochs}
\label{tab:cifar10_lr_acc}
\begin{tabular}{lcccc}
\toprule
\textbf{LR} & \textbf{LEHI (\%)} & \textbf{LEHIBRID (\%)} & \textbf{ADAM (\%)} & \textbf{ADAMW (\%)} \\
\midrule
$1 \times 10^{-4}$ & $74.804 \pm 0.904$ & $77.585 \pm 0.857$ & $87.139 \pm 0.965$ & $87.086 \pm 0.805$ \\
$3 \times 10^{-4}$ & $82.219 \pm 0.913$ & $83.539 \pm 1.031$ & $90.599 \pm 1.056$ & $90.219 \pm 0.764$ \\
$1 \times 10^{-3}$ & $86.161 \pm 2.029$ & $88.310 \pm 1.178$ & $91.455 \pm 0.999$ & $\mathbf{91.754 \pm 0.717}$ \\
$3 \times 10^{-3}$ & $90.336 \pm 0.823$ & $90.554 \pm 1.472$ & $\mathbf{91.653 \pm 0.775}$ & $91.445 \pm 0.491$ \\
$1 \times 10^{-2}$ & $91.736 \pm 0.695$ & $91.713 \pm 0.622$ & $91.283 \pm 0.637$ & $90.118 \pm 1.579$ \\
$3 \times 10^{-2}$ & $\mathbf{92.080 \pm 0.742}$ & $\mathbf{91.771 \pm 0.630}$ & $90.319 \pm 0.950$ & $87.693 \pm 1.380$ \\
$1 \times 10^{-1}$ & $91.468 \pm 1.325$ & $91.403 \pm 0.976$ & $89.610 \pm 1.784$ & $74.416 \pm 4.620$ \\
\bottomrule
\end{tabular}
\end{table}

\begin{table}[h!]
\centering
\caption{CIFAR-10 testing accuracy risk-adjusted selection scores ($c=2.0$)}
\label{tab:cifar10_acc_selection}
\begin{tabular}{lcccc}
\toprule
\textbf{LR} & \textbf{LEHI (\%)} & \textbf{LEHIBRID (\%)} & \textbf{ADAM (\%)} & \textbf{ADAMW (\%)} \\
\midrule
$1 \times 10^{-4}$ & 73.900 & 76.728 & 86.174 & 86.282 \\
$3 \times 10^{-4}$ & 81.306 & 82.508 & 89.543 & 89.455 \\
$1 \times 10^{-3}$ & 84.132 & 87.132 & 90.456 & \textbf{91.036} \\
$3 \times 10^{-3}$ & 89.514 & 89.082 & \textbf{90.877} & 90.954 \\
$1 \times 10^{-2}$ & 91.041 & 91.091 & 90.646 & 88.538 \\
$3 \times 10^{-2}$ & \textbf{91.338} & \textbf{91.141} & 89.368 & 86.312 \\
$1 \times 10^{-1}$ & 90.142 & 90.426 & 87.826 & 69.796 \\
\bottomrule
\end{tabular}
\end{table}

\subsubsection{CIFAR-100}
For CIFAR-100, we did a log-scale grid search over the set $\{1 \times 10^{-1}, 3 \times 10^{-2}, 1 \times 10^{-2}, 3 \times 10^{-3}, 1 \times 10^{-3}, 3 \times 10^{-4}, 1 \times 10^{-4}\}$. 
Table \ref{tab:cifar100_lr_acc} reports the mean testing accuracy and the associated stability metric ($2 \times$ standard deviation) calculated over the final $10$ epochs of training. The final risk-adjusted selection scores are reported in Table \ref{tab:cifar100_acc_selection}. We selected learning rates $\{3 \times 10^{-2}, 1 \times 10^{-2}\}$ for LEHI and LEHIBRID, and $\{3 \times 10^{-3}, 1 \times 10^{-3}, 3 \times 10^{-4}\}$ for Adam and AdamW to conduct further evaluations across two additional random seeds ($1$ and $2$).

\begin{table}[h!]
\centering
\small
\caption{CIFAR-100 mean testing accuracy $\pm \; 2 \times \text{standard deviation}$ across last $10$ epochs}
\label{tab:cifar100_lr_acc}
\begin{tabular}{lcccc}
\toprule
\textbf{LR} & \textbf{LEHI (\%)} & \textbf{LEHIBRID (\%)} & \textbf{ADAM (\%)} & \textbf{ADAMW (\%)} \\
\midrule
$1 \times 10^{-4}$ & $40.263 \pm 0.536$ & $43.702 \pm 0.438$ & $61.639 \pm 0.718$ & $62.051 \pm 0.946$ \\
$3 \times 10^{-4}$ & $51.005 \pm 0.345$ & $52.505 \pm 0.420$ & $67.538 \pm 1.290$ & $67.681 \pm 1.183$ \\
$1 \times 10^{-3}$ & $55.182 \pm 1.179$ & $56.809 \pm 0.864$ & $\mathbf{68.850 \pm 1.152}$ & $\mathbf{69.136 \pm 1.116}$ \\
$3 \times 10^{-3}$ & $62.225 \pm 1.654$ & $64.427 \pm 1.180$ & $67.865 \pm 0.904$ & $67.153 \pm 1.256$ \\
$1 \times 10^{-2}$ & $68.124 \pm 1.240$ & $\mathbf{69.507 \pm 0.557}$ & $64.966 \pm 0.710$ & $64.910 \pm 1.406$ \\
$3 \times 10^{-2}$ & $\mathbf{69.204 \pm 0.755}$ & $69.489 \pm 1.294$ & $64.488 \pm 0.827$ & $61.671 \pm 2.648$ \\
$1 \times 10^{-1}$ & $67.045 \pm 1.192$ & $66.795 \pm 1.076$ & $59.628 \pm 0.991$ & $38.621 \pm 6.785$ \\
\bottomrule
\end{tabular}
\end{table}

\begin{table}[h!]
\centering
\caption{CIFAR-100 testing accuracy risk-adjusted selection scores ($c=2.0$)}
\label{tab:cifar100_acc_selection}
\begin{tabular}{lcccc}
\toprule
\textbf{LR} & \textbf{LEHI} & \textbf{LEHIBRID} & \textbf{ADAM} & \textbf{ADAMW} \\
\midrule
$1 \times 10^{-4}$ & 39.727 & 43.264 & 60.921 & 61.105 \\
$3 \times 10^{-4}$ & 50.660 & 52.085 & 66.248 & 66.498 \\
$1 \times 10^{-3}$ & 54.003 & 55.945 & \textbf{67.698} & \textbf{68.020} \\
$3 \times 10^{-3}$ & 60.571 & 63.247 & 66.961 & 65.897 \\
$1 \times 10^{-2}$ & 66.884 & \textbf{68.950} & 64.256 & 63.504 \\
$3 \times 10^{-2}$ & \textbf{68.449} & 68.195 & 63.661 & 59.023 \\
$1 \times 10^{-1}$ & 65.853 & 65.719 & 58.637 & 31.836 \\
\bottomrule
\end{tabular}
\end{table}

\subsubsection{Penn Treebank}
For Penn Treebank, we did a log-scale grid search over the set $\{1 \times 10^{-1}, 3 \times 10^{-2}, 1 \times 10^{-2}, 3 \times 10^{-3}, 1 \times 10^{-3}, 3 \times 10^{-4}, 1 \times 10^{-4}, 3 \times 10^{-5}, 1 \times 10^{-5}\}$.
Table \ref{tab:ptb_lr_loss} reports the mean validation loss (log PPL) and the associated stability metric ($2 \times$ standard deviation) calculated over the final $3$ epochs of training. The final risk-adjusted selection scores are reported in Table \ref{tab:ptb_loss_selection}. We ran the whole set of learning rates for two additional random seeds ($1$ and $2$) to compare the learning rate sensitivity between LEHI, LEHIBRID, Adam, and AdamW. Eventually, after we selected the learning rates for each optimizer from the validation results, we evaluated on the testing set with random seeds $\{0,1,2\}$ as well.

\begin{table}[h!]
\centering
\footnotesize
\setlength{\tabcolsep}{3.0pt}
\renewcommand{\arraystretch}{1}
\caption{PTB mean validation loss (log perplexity) $\pm \; 2 \times \text{standard deviation}$ across last $3$ epochs}
\label{tab:ptb_lr_loss}
\begin{tabular}{lccccc}
\toprule
\textbf{LR} & \textbf{LEHI} & \textbf{LEHIBRID} & \textbf{ADAM} & \textbf{ADAMW} & \textbf{SOPHIA} \\
\midrule
$3 \times 10^{-6}$ & $8.557 \pm 0.055$ & $8.458 \pm 0.058$ & $5.312 \pm 0.056$ & $5.313 \pm 0.057$ & $4.458 \pm 0.061$ \\
$1 \times 10^{-5}$ & $6.981 \pm 0.047$ & $6.886 \pm 0.045$ & $4.382 \pm 0.039$ & $4.384 \pm 0.039$ & $\mathbf{3.955 \pm 0.020}$ \\
$3 \times 10^{-5}$ & $6.447 \pm 0.043$ & $6.334 \pm 0.049$ & $\mathbf{4.055 \pm 0.028}$ & $\mathbf{4.053 \pm 0.025}$ & $4.494 \pm 0.173$ \\
$1 \times 10^{-4}$ & $5.803 \pm 0.048$ & $5.667 \pm 0.051$ & $4.678 \pm 0.175$ & $4.616 \pm 0.164$ & $4.590 \pm 0.261$ \\
$3 \times 10^{-4}$ & $5.242 \pm 0.041$ & $5.127 \pm 0.040$ & $5.642 \pm 0.187$ & $5.485 \pm 0.189$ & $5.605 \pm 0.038$ \\
$1 \times 10^{-3}$ & $4.742 \pm 0.034$ & $4.635 \pm 0.034$ & $5.803 \pm 0.441$ & $5.629 \pm 0.333$ & $6.070 \pm 0.011$ \\
$3 \times 10^{-3}$ & $4.346 \pm 0.030$ & $4.245 \pm 0.034$ & $6.753 \pm 0.006$ & $4.815 \pm 0.328$ & $6.361 \pm 0.006$ \\
$1 \times 10^{-2}$ & $3.970 \pm 0.010$ & $4.007 \pm 0.045$ & $6.796 \pm 0.014$ & $7.560 \pm 0.114$ & $7.146 \pm 0.021$ \\
$3 \times 10^{-2}$ & $\mathbf{3.842 \pm 0.026}$ & $\mathbf{3.857 \pm 0.048}$ & $7.305 \pm 0.285$ & $7.326 \pm 0.309$ & $49.406 \pm 5.865$ \\
$1 \times 10^{-1}$ & $4.614 \pm 0.250$ & $4.985 \pm 0.338$ & $55.804 \pm 17.641$ & $7.787 \pm 0.035$ & $395.709 \pm 86.724$ \\
\bottomrule
\end{tabular}
\end{table}

\begin{table}[h!]
\centering
\caption{PTB validation loss (log perplexity) risk-adjusted selection scores ($c=2.0$)}
\label{tab:ptb_loss_selection}
\begin{tabular}{lccccc}
\toprule
\textbf{LR} & \textbf{LEHI} & \textbf{LEHIBRID} & \textbf{ADAM} & \textbf{ADAMW} & \textbf{SOPHIA} \\
\midrule
$3 \times 10^{-6}$ & 8.6116 & 8.5154 & 5.3679 & 5.3692 & 4.5188 \\
$1 \times 10^{-5}$ & 7.0278 & 6.9305 & 4.4210 & 4.4228 & \textbf{3.9749} \\
$3 \times 10^{-5}$ & 6.4906 & 6.3832 & \textbf{4.0826} & \textbf{4.0785} & 4.6672 \\
$1 \times 10^{-4}$ & 5.8515 & 5.7176 & 4.8531 & 4.7803 & 4.8511 \\
$3 \times 10^{-4}$ & 5.2830 & 5.1670 & 5.8290 & 5.6739 & 5.6428 \\
$1 \times 10^{-3}$ & 4.7762 & 4.6692 & 6.2445 & 5.9624 & 6.0816 \\
$3 \times 10^{-3}$ & 4.3764 & 4.2783 & 6.7595 & 5.1427 & 6.3673 \\
$1 \times 10^{-2}$ & 3.9798 & 4.0519 & 6.8095 & 7.6742 & 7.1663 \\
$3 \times 10^{-2}$ & \textbf{3.8678} & \textbf{3.9048} & 7.5899 & 7.6349 & 55.2716 \\
$1 \times 10^{-1}$ & 4.8640 & 5.3225 & 73.4445 & 7.8220 & 482.4331 \\
\bottomrule
\end{tabular}
\end{table}

\subsubsection{FineWeb-Edu}
For FineWeb-Edu, we first did a coarse log-scale grid search over the set $\{1 \times 10^{-1},  1 \times 10^{-2}, 1 \times 10^{-3}, 1 \times 10^{-4}, 1 \times 10^{-5}, 1 \times 10^{-6}, 1 \times 10^{-7}\}$. Then based on the results for the optimizers, we refined our search around the best learning rates from the coarse list: $\{3 \times 10^{-3}, 3 \times 10^{-4}\}$  for LEHI and LEHIBRID, $\{3 \times 10^{-2}, 3 \times 10^{-3}, 3 \times 10^{-6}, 3 \times 10^{-7}\}$ for Adam and AdamW, and $\{3 \times 10^{-2}, 3 \times 10^{-3}\}$ for Sophia-G. Table \ref{tab:fineweb-edu_lr_loss} reports the mean testing loss and the associated stability metric ($2 \times$ standard deviation) calculated over the final $3$ epochs of training. The final risk-adjusted selection scores are reported in Table \ref{tab:fineweb-edu_loss_selection}. We selected learning rates $\{3 \times 10^{-3}, 1 \times 10^{-3}\}$ for LEHI and LEHIBRID, $\{3 \times 10^{-3}, 3 \times 10^{-6}\}$ for Adam, $\{3 \times 10^{-2}, 3 \times 10^{-6}\}$ for AdamW, and $\{1 \times 10^{-2}, 3 \times 10^{-3}, 1 \times 10^{-3}\}$ for Sophia-G to conduct further evaluations across two additional random seeds ($1$ and $2$).

\begin{table}[h!]
\centering
\footnotesize
\setlength{\tabcolsep}{3.0pt}
\renewcommand{\arraystretch}{1}
\caption{FineWeb-Edu testing loss $\pm \; 2 \times \text{standard deviation}$ across last $300$ steps}
\label{tab:fineweb-edu_lr_loss}
\begin{tabular}{lccccc}
\toprule
\textbf{LR} & \textbf{LEHI} & \textbf{LEHIBRID} & \textbf{ADAM} & \textbf{ADAMW} & \textbf{SOPHIA} \\
\midrule
$1 \times 10^{-7}$ & $11.310 \pm 0.033$ & $11.310 \pm 0.033$ & $9.291 \pm 0.037$ & $9.291 \pm 0.037$ & $8.721 \pm 0.065$ \\
$3 \times 10^{-7}$ & -- & -- & $8.498 \pm 0.076$ & $8.498 \pm 0.076$ & -- \\
$1 \times 10^{-6}$ & $10.371 \pm 0.026$ & $10.371 \pm 0.026$ & $7.365 \pm 0.055$ & $7.366 \pm 0.055$ & $7.593 \pm 0.119$ \\
$3 \times 10^{-6}$ & -- & -- & $\mathbf{6.819 \pm 0.005}$ & $6.818 \pm 0.005$ & -- \\
$1 \times 10^{-5}$ & $9.508 \pm 0.030$ & $9.508 \pm 0.030$ & $7.885 \pm 0.140$ & $7.877 \pm 0.145$ & $7.959 \pm 0.040$ \\
$1 \times 10^{-4}$ & $7.843 \pm 0.054$ & $7.843 \pm 0.054$ & $7.803 \pm 0.037$ & $7.751 \pm 0.025$ & $7.305 \pm 0.343$ \\
$3 \times 10^{-4}$ & $7.192 \pm 0.035$ & $7.192 \pm 0.035$ & -- & -- & -- \\
$1 \times 10^{-3}$ & $\mathbf{6.652 \pm 0.021}$ & $\mathbf{6.652 \pm 0.022}$ & $7.991 \pm 0.306$ & $8.713 \pm 0.145$ & $7.053 \pm 0.056$ \\
$3 \times 10^{-3}$ & $6.909 \pm 0.145$ & $6.909 \pm 0.148$ & $6.942 \pm 0.084$ & $6.928 \pm 0.045$ & $7.042 \pm 0.033$ \\
$1 \times 10^{-2}$ & $8.238 \pm 0.067$ & $8.237 \pm 0.066$ & $6.910 \pm 0.219$ & $7.132 \pm 0.104$ & $\mathbf{6.996 \pm 0.107}$ \\
$3 \times 10^{-2}$ & -- & -- & $7.287 \pm 0.055$ & $\mathbf{6.515 \pm 0.087}$ & $54.587 \pm 25.593$ \\
$1 \times 10^{-1}$ & $8.418 \pm 0.065$ & $8.463 \pm 0.055$ & $10.990 \pm 4.329$ & $8.304 \pm 1.579$ & $412.060 \pm 76.470$ \\
\bottomrule
\end{tabular}
\end{table}

\begin{table}[h!]
\centering
\caption{FineWeb-Edu testing loss risk-adjusted selection scores ($k=2.0$)}
\label{tab:fineweb-edu_loss_selection}
\begin{tabular}{lccccc}
\toprule
\textbf{LR} & \textbf{LEHI} & \textbf{LEHIBRID} & \textbf{ADAM} & \textbf{ADAMW} & \textbf{SOPHIA} \\
\midrule
$1 \times 10^{-7}$ & 11.343 & 11.344 & 9.328 & 9.328 & 8.786 \\
$3 \times 10^{-7}$ & -- & -- & 8.574 & 8.574 & -- \\
$1 \times 10^{-6}$ & 10.396 & 10.396 & 7.421 & 7.421 & 7.712 \\
$3 \times 10^{-6}$ & -- & -- & \textbf{6.823} & 6.824 & -- \\
$1 \times 10^{-5}$ & 9.538 & 9.538 & 8.025 & 8.022 & 7.999 \\
$1 \times 10^{-4}$ & 7.897 & 7.897 & 7.839 & 7.776 & 7.648 \\
$3 \times 10^{-4}$ & 7.227 & 7.227 & -- & -- & -- \\
$1 \times 10^{-3}$ & \textbf{6.674} & \textbf{6.674} & 8.297 & 8.858 & 7.109 \\
$3 \times 10^{-3}$ & 7.054 & 7.057 & 7.027 & 6.973 & \textbf{7.075} \\
$1 \times 10^{-2}$ & 8.304 & 8.303 & 7.129 & 7.235 & 7.103 \\
$3 \times 10^{-2}$ & -- & -- & 7.342 & \textbf{6.602} & 80.180 \\
$1 \times 10^{-1}$ & 8.483 & 8.518 & 15.319 & 9.883 & 488.530 \\
\bottomrule
\end{tabular}
\end{table}

\subsection{Additional Results} \label{app.sub.add}
In this subsection, we provide the statistics for our experiments with $3$ random seeds. We averaged each experiment setting across $3$ random seeds, and calculated their average as well as standard deviation across the final several epochs.

\subsubsection{UCI Protein}
The summary of experiments is presented in Table \ref{tab:protein_summary}. We chose and plotted learning rate $3 \times 10^{-3}$ for LEHI and LEHIBRID, and $3 \times 10^{-4}$ for Adam and AdamW. Figure \ref{fig:protein_loss_raw} shows the raw plots without smoothing.

\begin{table}[h!]
\centering
\small
\caption{UCI Protein 3-seed summary (last $10$ epochs, $k=2.0$)}
\label{tab:protein_summary}
\begin{tabular}{lcccc}
\toprule
\textbf{Optimizer} & \textbf{LR} & \textbf{Mean Loss} & \textbf{2 $\times$ Std Dev} & \textbf{Combined Score} \\
\midrule
\textbf{LEHI} & $3.0 \times 10^{-3}$ & 0.2833 & 0.0006 & 0.2839 \\
              & $1.0 \times 10^{-2}$ & 0.2634 & 0.0021 & 0.2655 \\
              & $3.0 \times 10^{-2}$ & 0.2525 & 0.0051 & 0.2576 \\
              & $1.0 \times 10^{-1}$ & 0.2446 & 0.0061 & \textbf{0.2507} \\
\midrule
\textbf{LEHIBRID} & $3.0 \times 10^{-3}$ & 0.2764 & 0.0008 & 0.2772 \\
                  & $1.0 \times 10^{-2}$ & 0.2613 & 0.0038 & 0.2650 \\
                  & $3.0 \times 10^{-2}$ & 0.2472 & 0.0038 & 0.2510 \\
                  & $1.0 \times 10^{-1}$ & 0.2435 & 0.0060 & \textbf{0.2496} \\
\midrule
\textbf{ADAM} & $1.0 \times 10^{-4}$ & 0.2739 & 0.0007 & 0.2747 \\
              & $3.0 \times 10^{-4}$ & 0.2576 & 0.0013 & 0.2589 \\
              & $1.0 \times 10^{-3}$ & 0.2468 & 0.0053 & \textbf{0.2521} \\
              & $3.0 \times 10^{-3}$ & 0.2442 & 0.0090 & 0.2532 \\
              & $1.0 \times 10^{-2}$ & 0.2596 & 0.0076 & 0.2672 \\
\midrule
\textbf{ADAMW} & $1.0 \times 10^{-4}$ & 0.2750 & 0.0007 & 0.2757 \\
               & $3.0 \times 10^{-4}$ & 0.2596 & 0.0013 & 0.2609 \\
               & $1.0 \times 10^{-3}$ & 0.2516 & 0.0051 & \textbf{0.2567} \\
               & $3.0 \times 10^{-3}$ & 0.2530 & 0.0103 & 0.2633 \\
               & $1.0 \times 10^{-2}$ & 0.2736 & 0.0061 & 0.2798 \\
\bottomrule
\end{tabular}
\end{table}

\begin{figure}[h!]
  \centering
  \begin{tabular}{cc}
    \includegraphics[width=0.45\textwidth]{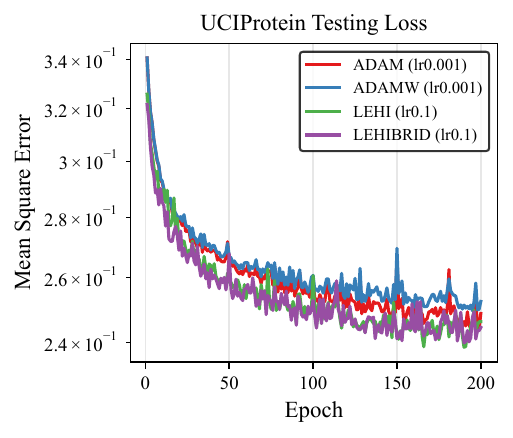} & 
    \includegraphics[width=0.45\textwidth]{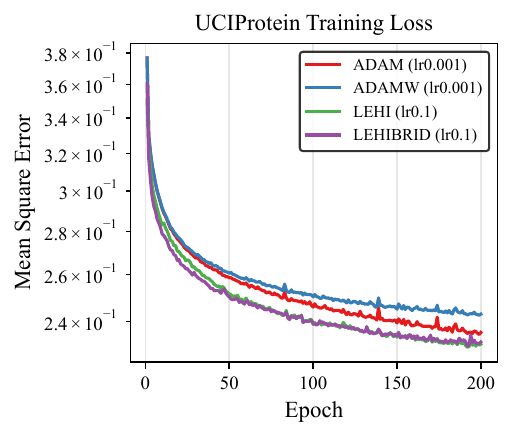} \\
    (a) Raw testing loss curves on UCI Protein  & (b) Raw training loss curves on UCI Protein
  \end{tabular}
  \caption{Raw testing and training loss curves on UCI Protein. The first epoch is omitted.}
  \label{fig:protein_loss_raw}
\end{figure}

\subsubsection{MNIST}
The summary of experiments is presented in Table \ref{tab:mnist_summary}. We chose and plotted learning rate $1 \times 10^{-1}$ for LEHI and LEHIBRID, and $1 \times 10^{-3}$ for Adam and AdamW. Figure \ref{fig:mnist_acc_raw} shows the raw plots without smoothing.

\begin{table}[h!]
\centering
\small
\caption{MNIST 3-seed summary (last $3$ epochs, $k=2.0$)}
\label{tab:mnist_summary}
\begin{tabular}{lcccc}
\toprule
\textbf{Optimizer} & \textbf{LR} & \textbf{Mean Acc (\%)} & \textbf{2 $\times$ Std Dev} & \textbf{Combined Score} \\
\midrule
\textbf{LEHI} & $3.0 \times 10^{-4}$ & 93.7500 & 0.2379 & 93.5121 \\
              & $1.0 \times 10^{-3}$ & 96.4256 & 0.1527 & 96.2729 \\
              & $3.0 \times 10^{-3}$ & 97.3389 & 0.0737 & \textbf{97.2652} \\
              & $1.0 \times 10^{-2}$ & 96.3144 & 0.3074 & 96.0070 \\
\midrule
\textbf{LEHIBRID} & $3.0 \times 10^{-4}$ & 94.5478 & 0.1833 & 94.3644 \\
                  & $1.0 \times 10^{-3}$ & 96.8744 & 0.1059 & 96.7685 \\
                  & $3.0 \times 10^{-3}$ & 97.2844 & 0.1639 & \textbf{97.1206} \\
                  & $1.0 \times 10^{-2}$ & 96.2889 & 0.0908 & 96.1981 \\
\midrule
\textbf{ADAM} & $1.0 \times 10^{-4}$ & 96.2356 & 0.1426 & 96.0930 \\
              & $3.0 \times 10^{-4}$ & 97.3089 & 0.0684 & \textbf{97.2405} \\
              & $1.0 \times 10^{-3}$ & 97.2611 & 0.1336 & 97.1276 \\
              & $3.0 \times 10^{-3}$ & 97.0511 & 0.2579 & 96.7932 \\
\midrule
\textbf{ADAMW} & $1.0 \times 10^{-4}$ & 96.2422 & 0.1424 & 96.0998 \\
               & $3.0 \times 10^{-4}$ & 97.2967 & 0.0467 & \textbf{97.2500} \\
               & $1.0 \times 10^{-3}$ & 97.3533 & 0.1067 & 97.2467 \\
               & $3.0 \times 10^{-3}$ & 97.0300 & 0.1737 & 96.8563 \\
\bottomrule
\end{tabular}
\end{table}

\begin{figure}[h!]
  \centering
  \begin{tabular}{cc}
    \includegraphics[width=0.45\textwidth]{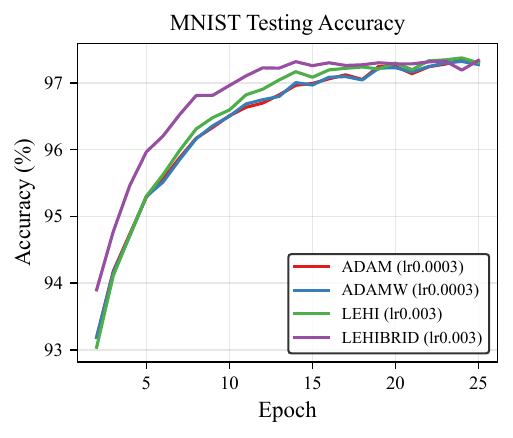} & 
    \includegraphics[width=0.45\textwidth]{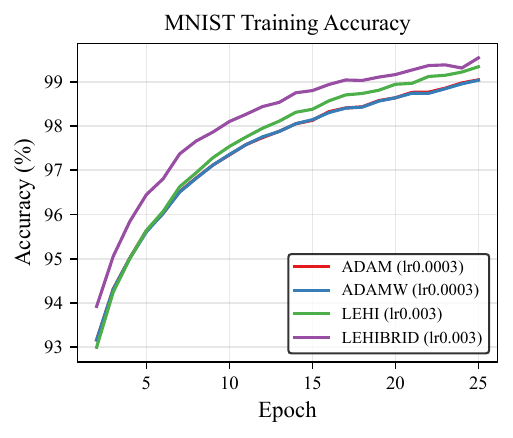} \\
    (a) Raw testing accuracy curves on MNIST  & (b) Raw training accuracy curves on MNIST
  \end{tabular}
  \caption{Raw testing and training accuracy curves on MNIST. The first $2$ epochs are omitted.}
  \label{fig:mnist_acc_raw}
\end{figure}

\subsubsection{CIFAR-10}
The summary of experiments is presented in Table \ref{tab:cifar10_summary}. We chose and plotted learning rates $3 \times 10^{-2}$ for LEHI and LEHIBRID, $1 \times 10^{-3}$ for Adam, and $3 \times 10^{-3}$ for AdamW. Figure \ref{fig:cifar10_acc_raw} shows the raw plots without smoothing.

\begin{table}[h!]
\centering
\small
\caption{CIFAR-10 3-seed summary (last $7$ epochs, $k=2.0$)}
\label{tab:cifar10_summary}
\begin{tabular}{lcccc}
\toprule
\textbf{Optimizer} & \textbf{LR} & \textbf{Mean Acc (\%)} & \textbf{2 $\times$ Std Dev} & \textbf{Combined Score} \\
\midrule
\textbf{LEHI} & $1.0 \times 10^{-2}$ & 91.6981 & 0.2278 & 91.4703 \\
              & $3.0 \times 10^{-2}$ & 91.9462 & 0.4449 & \textbf{91.5013} \\
\midrule
\textbf{LEHIBRID} & $1.0 \times 10^{-2}$ & 91.7581 & 0.5590 & 91.1991 \\
                  & $3.0 \times 10^{-2}$ & 91.8519 & 0.3323 & \textbf{91.5196} \\
\midrule
\textbf{ADAM} & $1.0 \times 10^{-3}$ & 91.6086 & 0.5680 & \textbf{91.0406} \\
              & $3.0 \times 10^{-3}$ & 91.6567 & 0.9279 & 90.7287 \\
\midrule
\textbf{ADAMW} & $1.0 \times 10^{-3}$ & 91.6033 & 0.5826 & 91.0207 \\
               & $3.0 \times 10^{-3}$ & 91.4690 & 0.1874 & \textbf{91.2816} \\
\bottomrule
\end{tabular}
\end{table}

\begin{figure}[h!]
  \centering
  \begin{tabular}{cc}
    \includegraphics[width=0.45\textwidth]{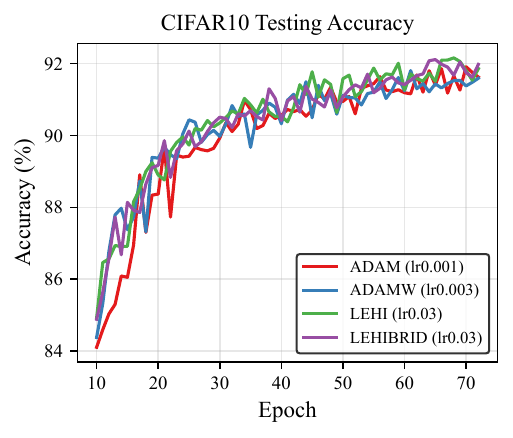} & 
    \includegraphics[width=0.45\textwidth]{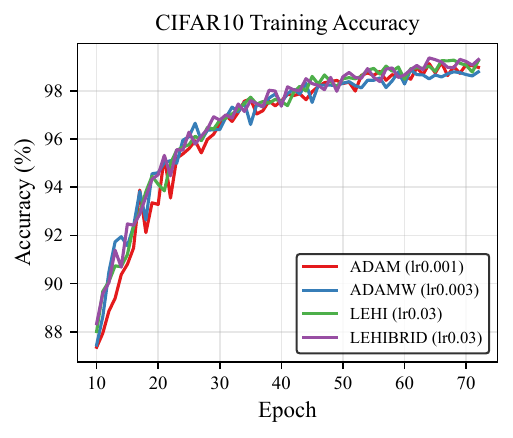} \\
    (a) Raw testing accuracy curves on CIFAR-10  & (b) Raw training accuracy curves on CIFAR-10
  \end{tabular}
  \caption{Raw testing and training accuracy curves on CIFAR-10. The first $10$ epochs are omitted.}
  \label{fig:cifar10_acc_raw}
\end{figure}

\subsubsection{CIFAR-100}
The summary of experiments is presented in Table \ref{tab:cifar100_summary}. We chose and plotted learning rates $3 \times 10^{-2}$ for LEHI, $1 \times 10^{-2}$ for LEHIBRID, $1 \times 10^{-3}$ for Adam and AdamW. Figure \ref{fig:cifar100_acc_raw} shows the raw plots without smoothing.

\begin{table}[h!]
\centering
\small
\caption{CIFAR-100 3-seed summary (last $10$ epochs, $k=2.0$)}
\label{tab:cifar100_summary}
\begin{tabular}{lcccc}
\toprule
\textbf{Optimizer} & \textbf{LR} & \textbf{Mean Acc (\%)} & \textbf{2 $\times$ Std Dev} & \textbf{Combined Score} \\
\midrule
\textbf{LEHI} & $1.0 \times 10^{-2}$ & 68.5713 & 0.5628 & 68.0085 \\
              & $3.0 \times 10^{-2}$ & 69.3520 & 0.8172 & \textbf{68.5348} \\
\midrule
\textbf{LEHIBRID} & $1.0 \times 10^{-2}$ & 69.4267 & 0.3150 & \textbf{69.1117} \\
                  & $3.0 \times 10^{-2}$ & 69.4640 & 0.5574 & 68.9066 \\
\midrule
\textbf{ADAM} & $3.0 \times 10^{-4}$ & 67.3887 & 0.5632 & 66.8254 \\
              & $1.0 \times 10^{-3}$ & 69.0363 & 0.5484 & \textbf{68.4879} \\
              & $3.0 \times 10^{-3}$ & 67.6883 & 0.6722 & 67.0162 \\
\midrule
\textbf{ADAMW} & $3.0 \times 10^{-4}$ & 67.5657 & 0.7810 & 66.7846 \\
               & $1.0 \times 10^{-3}$ & 69.0487 & 0.7393 & \textbf{68.3094} \\
               & $3.0 \times 10^{-3}$ & 67.1093 & 0.6594 & 66.4500 \\
\bottomrule
\end{tabular}
\end{table}

\begin{figure}[h!]
  \centering
  \begin{tabular}{cc}
    \includegraphics[width=0.45\textwidth]{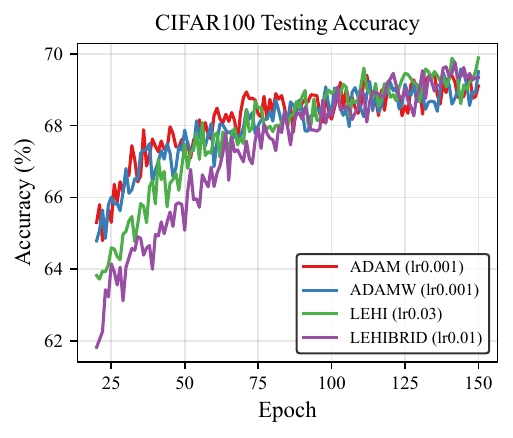} & 
    \includegraphics[width=0.45\textwidth]{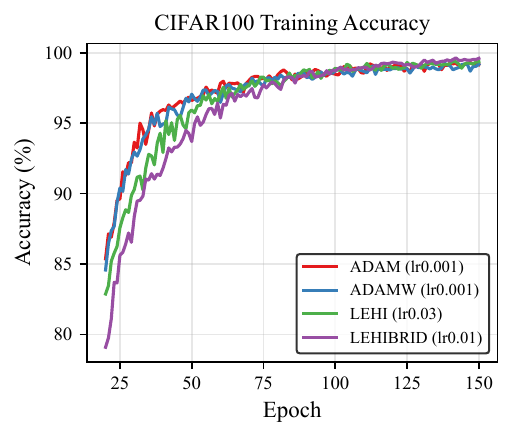} \\
    (a) Raw testing accuracy curves on CIFAR-100  & (b) Raw training accuracy curves on CIFAR-100
  \end{tabular}
  \caption{Raw testing and training accuracy curves on CIFAR-100. The first $20$ epochs are omitted.}
  \label{fig:cifar100_acc_raw}
\end{figure}

\subsubsection{Penn Treebank}
The summary of validation experiments is presented in Table \ref{tab:ptb_summary}. Based on these results, we evaluated the optimizers on the held-out testing set using learning rates $3 \times 10^{-2}$ for LEHI, $1 \times 10^{-2}$ for LEHIBRID, and $3 \times 10^{-5}$ for Adam and AdamW with random seeds $\{0,1,2\}$. Table \ref{tab:ptb_test_summary} shows the summary of testing performance, and Figure \ref{fig:ptb_loss_raw} shows the raw plots without smoothing. Additionally, we provide the comparison between validation loss and testing loss using the same learning rates in Figure \ref{fig:ptb_valid_raw}; they closely match each other and this maintained similarity demonstrates that our hyperparameter choices did not lead to overfitting to the validation set.

\begin{table}[h!]
\centering
\small
\caption{Penn Treebank 3-seed validation summary (last $3$ epochs, $k=2.0$)}
\label{tab:ptb_summary}
\begin{tabular}{lcccc}
\toprule
\textbf{Optimizer} & \textbf{LR} & \textbf{Mean Loss} & \textbf{2 $\times$ Std Dev} & \textbf{Combined Score} \\
\midrule
\textbf{LEHI} & $3.0 \times 10^{-6}$ & 8.5298 & 0.0567 & 8.5865 \\
              & $1.0 \times 10^{-5}$ & 6.9801 & 0.0450 & 7.0251 \\
              & $3.0 \times 10^{-5}$ & 6.4496 & 0.0447 & 6.4942 \\
              & $1.0 \times 10^{-4}$ & 5.7972 & 0.0488 & 5.8460 \\
              & $3.0 \times 10^{-4}$ & 5.2401 & 0.0402 & 5.2803 \\
              & $1.0 \times 10^{-3}$ & 4.7461 & 0.0348 & 4.7809 \\
              & $3.0 \times 10^{-3}$ & 4.3444 & 0.0318 & 4.3762 \\
              & $1.0 \times 10^{-2}$ & 3.9623 & 0.0068 & 3.9690 \\
              & $3.0 \times 10^{-2}$ & 3.8619 & 0.0299 & \textbf{3.8918} \\
              & $1.0 \times 10^{-1}$ & 4.6245 & 0.2623 & 4.8868 \\
\midrule
\textbf{LEHIBRID} & $3.0 \times 10^{-6}$ & 8.4302 & 0.0616 & 8.4918 \\
                  & $1.0 \times 10^{-5}$ & 6.8887 & 0.0440 & 6.9328 \\
                  & $3.0 \times 10^{-5}$ & 6.3332 & 0.0500 & 6.3832 \\
                  & $1.0 \times 10^{-4}$ & 5.6612 & 0.0496 & 5.7107 \\
                  & $3.0 \times 10^{-4}$ & 5.1271 & 0.0393 & 5.1663 \\
                  & $1.0 \times 10^{-3}$ & 4.6384 & 0.0350 & 4.6734 \\
                  & $3.0 \times 10^{-3}$ & 4.2379 & 0.0362 & 4.2740 \\
                  & $1.0 \times 10^{-2}$ & 3.9643 & 0.0247 & \textbf{3.9890} \\
                  & $3.0 \times 10^{-2}$ & 3.9093 & 0.0817 & 3.9910 \\
                  & $1.0 \times 10^{-1}$ & 4.9625 & 0.3163 & 5.2788 \\
\midrule
\textbf{ADAM} & $3.0 \times 10^{-6}$ & 5.3104 & 0.0559 & 5.3663 \\
              & $1.0 \times 10^{-5}$ & 4.3844 & 0.0376 & 4.4220 \\
              & $3.0 \times 10^{-5}$ & 4.0529 & 0.0225 & \textbf{4.0754} \\
              & $1.0 \times 10^{-4}$ & 4.6832 & 0.1862 & 4.8695 \\
              & $3.0 \times 10^{-4}$ & 5.6280 & 0.2005 & 5.8285 \\
              & $1.0 \times 10^{-3}$ & 5.9294 & 0.4243 & 6.3538 \\
              & $3.0 \times 10^{-3}$ & 6.8739 & 0.0042 & 6.8781 \\
              & $1.0 \times 10^{-2}$ & 6.7959 & 0.0136 & 6.8095 \\
              & $3.0 \times 10^{-2}$ & 7.5288 & 1.0875 & 8.6164 \\
              & $1.0 \times 10^{-1}$ & 55.8035 & 17.6410 & 73.4445 \\
\midrule
\textbf{ADAMW} & $3.0 \times 10^{-6}$ & 5.3112 & 0.0564 & 5.3677 \\
               & $1.0 \times 10^{-5}$ & 4.3863 & 0.0376 & 4.4239 \\
               & $3.0 \times 10^{-5}$ & 4.0513 & 0.0205 & \textbf{4.0718} \\
               & $1.0 \times 10^{-4}$ & 4.6226 & 0.1765 & 4.7991 \\
               & $3.0 \times 10^{-4}$ & 5.4602 & 0.1950 & 5.6552 \\
               & $1.0 \times 10^{-3}$ & 5.6664 & 0.3169 & 5.9834 \\
               & $3.0 \times 10^{-3}$ & 4.7419 & 0.3175 & 5.0594 \\
               & $1.0 \times 10^{-2}$ & 7.4087 & 0.0830 & 7.4916 \\
               & $3.0 \times 10^{-2}$ & 9.1674 & 0.0033 & 9.1708 \\
               & $1.0 \times 10^{-1}$ & 8.4660 & 0.0691 & 8.5351 \\
\midrule
\textbf{SOPHIA} & $3.0 \times 10^{-6}$ & 4.4604 & 0.0646 & 4.5250 \\
                & $1.0 \times 10^{-5}$ & 3.9546 & 0.0169 & \textbf{3.9715} \\
                & $3.0 \times 10^{-5}$ & 4.5265 & 0.1724 & 4.6990 \\
                & $1.0 \times 10^{-4}$ & 4.5664 & 0.2477 & 4.8141 \\
                & $3.0 \times 10^{-4}$ & 5.5557 & 0.0112 & 5.5668 \\
                & $1.0 \times 10^{-3}$ & 5.9687 & 0.0217 & 5.9904 \\
                & $3.0 \times 10^{-3}$ & 6.2818 & 0.1180 & 6.3998 \\
                & $1.0 \times 10^{-2}$ & 7.0339 & 0.0104 & 7.0443 \\
                & $3.0 \times 10^{-2}$ & 68.9475 & 13.0007 & 81.9483 \\
                & $1.0 \times 10^{-1}$ & 399.8266 & 3.8071 & 403.6337 \\
\bottomrule
\end{tabular}
\end{table}

\begin{table}[ht]
\centering
\small
\caption{Penn Treebank 3-seed testing summary (last $3$ epochs, $k=2.0$)}
\label{tab:ptb_test_summary}
\begin{tabular}{lcccc}
\toprule
\textbf{Optimizer} & \textbf{LR} & \textbf{Mean Loss} & \textbf{2 $\times$ Std Dev} & \textbf{Combined Score} \\
\midrule
\textbf{LEHI}     & $3.0 \times 10^{-2}$ & 3.7325 & 0.0301 & \textbf{3.7626} \\
\midrule
\textbf{SOPHIA} & $1.0 \times 10^{-5}$ & 3.8237 & 0.0173 & \textbf{3.8410} \\
\midrule
\textbf{LEHIBRID} & $1.0 \times 10^{-2}$ & 3.8283 & 0.0245 & \textbf{3.8528} \\
\midrule
\textbf{ADAMW}    & $3.0 \times 10^{-5}$ & 3.9150 & 0.0215 & \textbf{3.9365} \\
\midrule
\textbf{ADAM}     & $3.0 \times 10^{-5}$ & 3.9168 & 0.0237 & \textbf{3.9405} \\
\bottomrule
\end{tabular}
\end{table}

\begin{figure}[h!]
  \centering
  \begin{tabular}{cc}
    \includegraphics[width=0.45\textwidth]{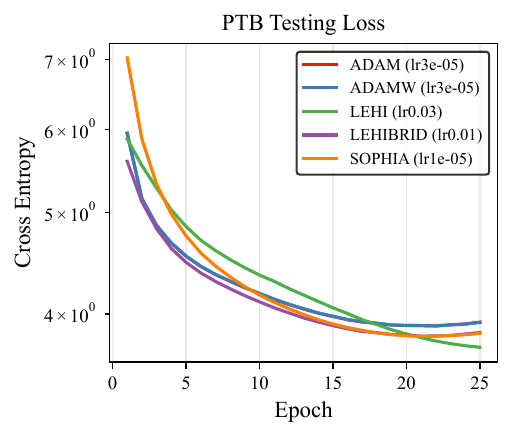} & 
    \includegraphics[width=0.45\textwidth]{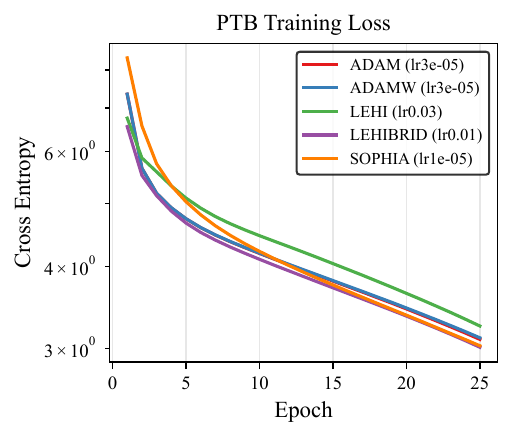} \\
    (a) Raw testing loss curves on Penn Treebank  & (b) Raw training loss curves on Penn Treebank
  \end{tabular}
  \caption{Raw testing and training loss curves on Penn Treebank.}
  \label{fig:ptb_loss_raw}
\end{figure}

\begin{figure}[h!]
  \centering
  \begin{tabular}{cc}
    \includegraphics[width=0.45\textwidth]{figures/3_seed_results/PTB_test_loss_raw.pdf} & 
    \includegraphics[width=0.45\textwidth]{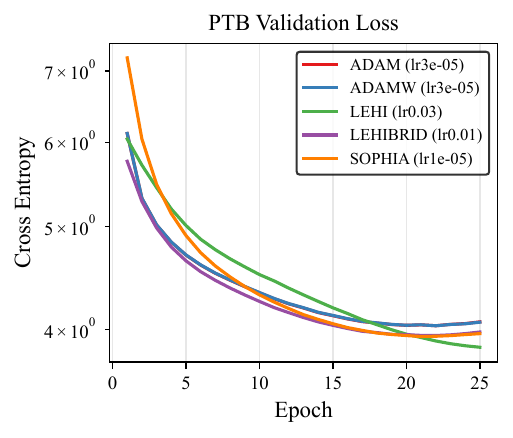} \\
    (a) Raw testing loss curves on Penn Treebank  & (b) Raw validation loss curves on Penn Treebank
  \end{tabular}
  \caption{Raw testing and validation loss curves on Penn Treebank.}
  \label{fig:ptb_valid_raw}
\end{figure}

\subsubsection{FineWeb-Edu}
The summary of experiments is presented in Table \ref{tab:fineweb-edu_summary}. We chose and plotted learning rates $1 \times 10^{-3}$ for LEHI and LEHIBRID, $3 \times 10^{-6}$ for Adam and AdamW, and $3 \times 10^{-3}$ for Sophia-G. Figure \ref{fig:fineweb-edu_loss_raw} shows the testing loss without smoothing.  Training loss not included since the evaluation was based on steps.

\begin{table}[h!]
\centering
\small
\caption{FineWeb-Edu 3-seed summary (last 300 steps, $k=2.0$)}
\label{tab:fineweb-edu_summary}
\begin{tabular}{lcccc}
\toprule
\textbf{Optimizer} & \textbf{LR} & \textbf{Mean Loss} & \textbf{2 $\times$ Std Dev} & \textbf{Combined Score} \\
\midrule
\textbf{LEHI} & $1.0 \times 10^{-3}$ & 6.6491 & 0.0189 & \textbf{6.6680} \\
              & $3.0 \times 10^{-3}$ & 6.9125 & 0.1465 & 7.0590 \\
\midrule
\textbf{LEHIBRID} & $1.0 \times 10^{-3}$ & 6.6491 & 0.0190 & \textbf{6.6681} \\
                  & $3.0 \times 10^{-3}$ & 6.9138 & 0.1493 & 7.0631 \\
\midrule
\textbf{ADAM} & $3.0 \times 10^{-6}$ & 6.8197 & 0.0014 & \textbf{6.8212} \\
              & $3.0 \times 10^{-3}$ & 7.0084 & 0.0538 & 7.0621 \\
\midrule
\textbf{ADAMW} & $3.0 \times 10^{-6}$ & 6.8192 & 0.0012 & \textbf{6.8205} \\
               & $3.0 \times 10^{-3}$ & 6.9325 & 0.0416 & 6.9741 \\
               & $3.0 \times 10^{-2}$ & 7.1038 & 0.1195 & 7.2234 \\
\midrule
\textbf{SOPHIA} & $1.0 \times 10^{-3}$ & 7.0450 & 0.0390 & 7.0840 \\
                & $3.0 \times 10^{-3}$ & 7.0345 & 0.0283 & \textbf{7.0629} \\
                & $1.0 \times 10^{-2}$ & 7.0752 & 0.0630 & 7.1381 \\
\bottomrule
\end{tabular}
\end{table}

\begin{figure}[h!]
    \centering
    \includegraphics[width=0.4\columnwidth]{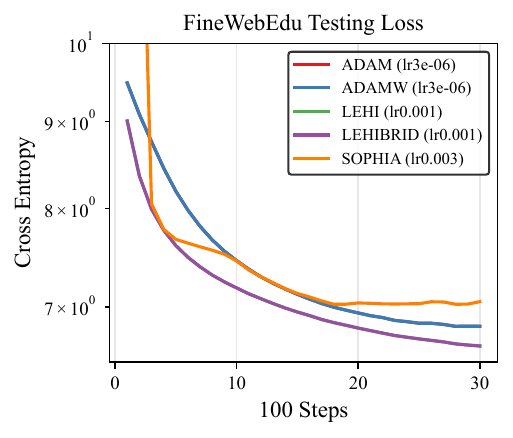}
    \caption{Raw testing loss curves on FineWeb-Edu}
    \label{fig:fineweb-edu_loss_raw}
\end{figure}

\subsection{Sensitivity Analysis on Penn Treebank}
Table \ref{tab:ptb_stability_summary} shows more extensive results of the sensitivity analysis on learning rates. The column Max Grad Seen records the largest infinity norm of loss function seen across $3$ runs. The column Avg. Spikes records the average number of infinity norm of loss function larger than $10$ within $25$ epochs across $3$ runs. And the column NaN Failures records the number of runs ending at NaN out of $3$. Finally, the column Status summarizes the results: if an optimizer has NaN issues at certain learning rate, it will be marked as FAILED; if no NaN issues, but there are spikes, it will be marked as NOISY; otherwise, the optimizer-learning rate pair is marked as STABLE. One can observe that both LEHI and LEHIBRID are stable across all learning rates searched with all $3$ random seeds, while Adam and AdamW are much more sensitive to learning rate choices.

\begin{table}[h!]
\centering
\small
\caption{Penn Treebank training stability and gradient robustness summary (3 seeds)}
\label{tab:ptb_stability_summary}
\begin{tabular}{lccccc}
\toprule
\textbf{Optimizer} & \textbf{LR} & \textbf{Max Grad Seen} & \textbf{Avg. Spikes} & \textbf{NaN Failures} & \textbf{Status} \\
\midrule
\textbf{ADAM} & $3 \times 10^{-6}$ & $1.23 \times 10^{-1}$ & 0.00 & 0 & STABLE \\
              & $1 \times 10^{-5}$ & $9.33 \times 10^{-2}$ & 0.00 & 0 & STABLE \\
              & $3 \times 10^{-5}$ & $7.74 \times 10^{-2}$ & 0.00 & 0 & STABLE \\
              & $1 \times 10^{-4}$ & $6.75 \times 10^{-2}$ & 0.00 & 0 & STABLE \\
              & $3 \times 10^{-4}$ & $5.21 \times 10^{-2}$ & 0.00 & 0 & STABLE \\
              & $1 \times 10^{-3}$ & $5.21 \times 10^{-2}$ & 0.00 & 0 & STABLE \\
              & $3 \times 10^{-3}$ & $4.40 \times 10^{32}$ & 19.33 & 0 & NOISY \\
              & $1 \times 10^{-2}$ & $1.02 \times 10^{32}$ & 10.00 & 2 & FAILED \\
              & $3 \times 10^{-2}$ & $4.26 \times 10^{32}$ & 19.67 & 1 & FAILED \\
              & $1 \times 10^{-1}$ & $1.91 \times 10^{29}$ & 11.67 & 2 & FAILED \\
\midrule
\textbf{ADAMW} & $3 \times 10^{-6}$ & $1.40 \times 10^{-1}$ & 0.00 & 0 & STABLE \\
               & $1 \times 10^{-5}$ & $9.33 \times 10^{-2}$ & 0.00 & 0 & STABLE \\
               & $3 \times 10^{-5}$ & $7.79 \times 10^{-2}$ & 0.00 & 0 & STABLE \\
               & $1 \times 10^{-4}$ & $7.09 \times 10^{-2}$ & 0.00 & 0 & STABLE \\
               & $3 \times 10^{-4}$ & $5.24 \times 10^{-2}$ & 0.00 & 0 & STABLE \\
               & $1 \times 10^{-3}$ & $4.82 \times 10^{-2}$ & 0.00 & 0 & STABLE \\
               & $3 \times 10^{-3}$ & $6.74 \times 10^{-2}$ & 0.00 & 0 & STABLE \\
               & $1 \times 10^{-2}$ & $8.00 \times 10^{27}$ & 15.33 & 1 & FAILED \\
               & $3 \times 10^{-2}$ & $1.96 \times 10^{30}$ & 5.33 & 1 & FAILED \\
               & $1 \times 10^{-1}$ & $2.35 \times 10^{17}$ & 1.33 & 0 & NOISY \\
\midrule
\textbf{LEHI} & $3 \times 10^{-6}$ & $1.45 \times 10^{-1}$ & 0.00 & 0 & STABLE \\
              & $1 \times 10^{-5}$ & $1.44 \times 10^{-1}$ & 0.00 & 0 & STABLE \\
              & $3 \times 10^{-5}$ & $1.33 \times 10^{-1}$ & 0.00 & 0 & STABLE \\
              & $1 \times 10^{-4}$ & $1.29 \times 10^{-1}$ & 0.00 & 0 & STABLE \\
              & $3 \times 10^{-4}$ & $1.19 \times 10^{-1}$ & 0.00 & 0 & STABLE \\
              & $1 \times 10^{-3}$ & $9.96 \times 10^{-2}$ & 0.00 & 0 & STABLE \\
              & $3 \times 10^{-3}$ & $8.23 \times 10^{-2}$ & 0.00 & 0 & STABLE \\
              & $1 \times 10^{-2}$ & $6.40 \times 10^{-2}$ & 0.00 & 0 & STABLE \\
              & $3 \times 10^{-2}$ & $7.34 \times 10^{-2}$ & 0.00 & 0 & STABLE \\
              & $1 \times 10^{-1}$ & $\mathbf{3.23 \times 10^{-2}}$ & 0.00 & 0 & STABLE \\
\midrule
\textbf{LEHIBRID} & $3 \times 10^{-6}$ & $1.39 \times 10^{-1}$ & 0.00 & 0 & STABLE \\
                  & $1 \times 10^{-5}$ & $1.43 \times 10^{-1}$ & 0.00 & 0 & STABLE \\
                  & $3 \times 10^{-5}$ & $1.32 \times 10^{-1}$ & 0.00 & 0 & STABLE \\
                  & $1 \times 10^{-4}$ & $1.26 \times 10^{-1}$ & 0.00 & 0 & STABLE \\
                  & $3 \times 10^{-4}$ & $1.14 \times 10^{-1}$ & 0.00 & 0 & STABLE \\
                  & $1 \times 10^{-3}$ & $9.48 \times 10^{-2}$ & 0.00 & 0 & STABLE \\
                  & $3 \times 10^{-3}$ & $7.62 \times 10^{-2}$ & 0.00 & 0 & STABLE \\
                  & $1 \times 10^{-2}$ & $6.24 \times 10^{-2}$ & 0.00 & 0 & STABLE \\
                  & $3 \times 10^{-2}$ & $6.83 \times 10^{-2}$ & 0.00 & 0 & STABLE \\
                  & $1 \times 10^{-1}$ & $\mathbf{4.73 \times 10^{-2}}$ & 0.00 & 0 & STABLE \\
\midrule
\textbf{SOPHIA} & $3 \times 10^{-6}$ & $1.23 \times 10^{-1}$ & 0.00 & 0 & STABLE \\
                & $1 \times 10^{-5}$ & $1.02 \times 10^{-1}$ & 0.00 & 0 & STABLE \\
                & $3 \times 10^{-5}$ & $6.13 \times 10^{-2}$ & 0.00 & 0 & STABLE \\
                & $1 \times 10^{-4}$ & $4.29 \times 10^{-2}$ & 0.00 & 0 & STABLE \\
                & $3 \times 10^{-4}$ & $9.35 \times 10^{0}$ & 0.00 & 0 & STABLE \\
                & $1 \times 10^{-3}$ & $1.14 \times 10^{1}$ & 0.33 & 0 & NOISY \\
                & $3 \times 10^{-3}$ & $6.10 \times 10^{1}$ & 1.33 & 0 & NOISY \\
                & $1 \times 10^{-2}$ & $3.45 \times 10^{6}$ & 5.67 & 0 & NOISY \\
                & $3 \times 10^{-2}$ & $1.20 \times 10^{10}$ & 11.33 & 0 & NOISY \\
                & $1 \times 10^{-1}$ & $1.77 \times 10^{22}$ & 11.33 & 0 & NOISY \\
\bottomrule
\end{tabular}
\end{table}



\end{document}